\documentclass[12pt,fleqn]{article}
\usepackage{verbatim}

\usepackage{enumitem}
\usepackage{hyphenat}
\usepackage{a4}
\usepackage{amsthm}
\usepackage{tikz}
\usepackage{ifthen}
\usepackage{amsmath}
\usepackage{amssymb}
\usepackage{mathtools}
\usepackage{graphicx}
\usepackage{color}
\usepackage{lscape}
\usepackage{stmaryrd} 
\usepackage{textcomp} 
\usepackage{url}

\newtheoremstyle{myplain}
  {}
  {}
  {\itshape}
  {}
  {\bfseries}
  {.}
  { }
  {\thmnumber{#2}\thmname{ #1}\thmnote{ (#3)}}

\newtheoremstyle{mydefinition}
  {}
  {}
  {}
  {}
  {\bfseries}
  {.}
  { }
  {\thmnumber{#2}\thmname{ #1}\thmnote{ (#3)}}

\theoremstyle{myplain}
\newtheorem{theorem}{Theorem}[section]
\newtheorem{lemma}[theorem]{Lemma}
\newtheorem{proposition}[theorem]{Proposition}

\theoremstyle{mydefinition}
\newtheorem{definition}[theorem]{Definition}
\newtheorem{example}[theorem]{Example}
\newtheorem{examples}[theorem]{Examples}

\newtheorem{remark}[theorem]{Remark}
\newtheorem{problem}[theorem]{{\sc Problem}}

\newtheorem{notation}[theorem]{Notation}

\numberwithin{equation}{theorem} 

\newcommand{\New}[1]{\emph{#1}}

\DeclareMathOperator{\Pol}{Pol}
\DeclareMathOperator{\Inv}{Inv}
\DeclareMathOperator{\Op}{Op}
\newcommand{\Opa}[1][]{\Op^{(#1)}}
\DeclareMathOperator{\Rel}{Rel}
\newcommand{\Rela}[1][]{\Rel^{(#1)}}

\DeclareMathOperator{\SOp}{\prescript{S}{}Op}
\DeclareMathOperator{\SpOp}{\prescript{S'}{}Op}
\DeclareMathOperator{\ShOp}{\prescript{\widehat{S}}{}Op}
\newcommand{\SOpa}[1][]{\SOp^{(#1)}}
\DeclareMathOperator{\SRel}{\prescript{S}{}Rel}

\newcommand{\SRela}[1][]{\SRel^{(#1)}}

\newcommand{\SJ}{\prescript{S}{}\!\!J} 
\newcommand{\SD}{\prescript{S}{}\!D} 

\newcommand{\SSg}[2][]{\ensuremath{\prescript{S}{}\!\langle #2 \rangle_{#1}}}
\newcommand{\SRg}[2][]{\ensuremath{\prescript{S}{}\![#2]_{#1}}}
\newcommand{\SL}{\prescript{S}{}\!\cL}

\DeclareMathOperator{\SPol}{\prescript{S}{}Pol}
\DeclareMathOperator{\SpPol}{\prescript{S^{\prime}}{}Pol}
\DeclareMathOperator{\ShPol}{\prescript{\widehat{S}}{}Pol}
\DeclareMathOperator{\SInv}{\prescript{S}{}Inv}
\newcommand{\SInva}[1][]{\SInv^{(#1)}} 

\newcommand{\Spreserves}{\mathrel{\stackrel{S}{\triangleright}}}
\DeclareMathOperator{\sgn}{sgn} 
\DeclareMathOperator{\Sgn}{Sgn}

\DeclareMathOperator{\ar}{ar} 

\DeclareMathOperator{\pr}{pr}
\newcommand{\preserves}{\mathrel{\triangleright}}
\newcommand{\notpreserves}{\mathrel{\not\triangleright}}
\newcommand{\id}{{\sf id}}
\newcommand{\Int}[2][\mathsf{\Quord(A)}]{[#2]_{#1}}

\newcommand{\cL}{\mathcal{L}}
\newcommand{\cM}{\mathcal{M}}
\newcommand{\cT}{\mathcal{T}}

\newcommand{\lambdahat}{\widehat{\lambda}}
\newcommand{\kappahat}{\widehat{\kappa}}

\newcommand{\rhobar}{\overline{\rho}}

\newcommand{\blambda}{\boldsymbol{\lambda}}
\newcommand{\bfr}{\mathbf{r}}

\newcommand{\N}{\mathbb{N}} 

\newcommand{\Sg}[2][]{\ensuremath{\langle #2 \rangle_{#1}}}
\newcommand{\Rg}[2][]{\ensuremath{[#2]_{#1}}}

\DeclareMathOperator{\Fun}{Fun}
\DeclareMathOperator{\RelP}{RelP}
\DeclareMathOperator{\PolP}{PolP}
\DeclareMathOperator{\InvP}{InvP}

\let\rho=\varrho
\let\epsilon=\varepsilon
\let\phi=\varphi
\let\kappa=\varkappa

\author{Peter Jipsen\\ Faculty of Mathematics
  \\ Chapman University, Orange, CA (USA)
\and Erkko Lehtonen\\
Department of Mathematics \\
Khalifa University \\
Abu Dhabi (United Arab Emirates)
\and Reinhard P\"oschel%
 \\Institute of Algebra\\Faculty of Mathematics\\TU Dresden (Germany)}
\date{\small version July 31, 2023}

\title{$S$-preclones and the Galois connection $\SPol$--$\SInv$, Part I}

\begin{document}

\maketitle

\begin{abstract}
  We consider \emph{$S$\hyp{}operations} $f \colon A^{n} \to A$ in
  which each argument is assigned a \emph{signum} $s \in S$
  representing a ``property'' such as being order\hyp{}preserving or
  order\hyp{}reversing with respect to a fixed partial order on
  $A$. The set $S$ of such properties is assumed to have a monoid
  structure reflecting the behaviour of these properties under the
  composition of $S$\hyp{}operations (e.g., order\hyp{}reversing
  composed with order\hyp{}reversing is order\hyp{}preserving). The
  collection of all $S$\hyp{}operations with prescribed properties for
  their signed arguments is not a clone (since it is not closed under
  arbitrary identification of arguments), but it is a preclone with
  special properties, which leads to the notion of
  \emph{$S$\hyp{}preclone}. We introduce \emph{$S$\hyp{}relations}
  $\varrho = (\varrho_{s})_{s \in S}$, \emph{$S$\hyp{}relational
    clones}, and a preservation property ($f \Spreserves \varrho$),
  and we consider the induced Galois connection $\SPol$--$\SInv$. The
  $S$\hyp{}preclones and $S$\hyp{}relational clones turn out to be
  exactly the closed sets of this Galois connection. We also establish
  some basic facts about the structure of the lattice of all
  $S$\hyp{}preclones on $A$.
\end{abstract}


\section{Introduction}

Clones of operations have been studied since the seminal publication
by Emil Post \cite{Pos21} (announced in \cite{Pos20}, in full detail
in \cite{Pos41}) and are fundamental algebraic objects in universal
algebra. They have been generalized to Lawvere theories and preclones
(i.e., operads), and investigations into the lattice of clones on a
finite set have been greatly illuminated by the Pol--Inv Galois
connection between clones and relational clones. The collection of all
order\hyp{}preserving operations on a poset form a clone, and theories
of ordered algebras are classified by subclones of this type. In
algebraic logic formulas (i.e., elements of the algebras) are ordered
by a consequence relation, and logical operations such as negation and
implication are not order\hyp{}preserving in all arguments. They are
however \emph{order\hyp{}preserving or order\hyp{}reversing} in each
argument. The collection of all such operations on a nontrivial poset
(our ``motivating example'' presented in more detail as
Example~\ref{Ex0}) is not a clone since it is not closed under
arbitrary identification of arguments, but it is a preclone.

In this paper we study a generalization of such preclones, so-called
\New{$S$\hyp{}preclones}. On one hand they provide a classification of
the generating sets of operations for partially ordered algebras. But
on the other hand these new algebraic structures of $S$\hyp{}preclones
are applicable to situations where the arguments of an operation
should have different ``properties'' $s\in S$ ``expressible'' by
relations (like order\hyp{}preserving and order\hyp{}reversing via
partial order relations). The set $S$ of ``properties'' (called
\New{signa}) itself has an algebraic structure which reflects
composition of operations, i.e., we need $S$ to be a monoid
(e.g. order\hyp{}reversing composed with order\hyp{}reversing gives
order\hyp{}preserving, thus $S$ is a two\hyp{}element group for our
motivating example, see Example~\ref{Ex0}). A universal algebraic
perspective on partially ordered algebras can be found in
\cite{Pig2004}.

Our results in the current paper (Part I) are about preclones of
$S$\hyp{}operations on a \emph{finite} base set $A$. In a subsequent
paper (Part II) we consider in detail the case of \emph{Boolean
  $S$\hyp{}preclones}, where $A$ is a 2\hyp{}element set.

In Section~\ref{secC} we introduce the basic notions of an
$S$\hyp{}operation, $S$\hyp{}preclone, $S$\hyp{}relation and
$S$\hyp{}relational clone which generalize the classical notions in a
natural way. In Section~\ref{secD} the Galois connection
$\SPol$--$\SInv$ induced by the crucial property
``\New{$S$\hyp{}operation $S$\hyp{}preserves $S$\hyp{}relation}'',
$f\Spreserves\rho$, is defined. It is shown that the Galois closed
sets of $S$\hyp{}operations form an $S$\hyp{}preclone and likewise the
Galois closed sets of $S$\hyp{}relations form an $S$\hyp{}relational
clone.

Section~\ref{secE} contains the central result that for a finite set
$A$ and a set $F$ of $S$\hyp{}operations on $A$, the set
$\SPol\SInv F$ is the smallest $S$\hyp{}preclone containing $F$, hence
all $S$\hyp{}preclones are Galois closed (Theorem~\ref{Thm-I}). On the
relational side the corresponding result holds (Theorem~\ref{Thm-II}):
for a set $Q$ of $S$\hyp{}relations, $\SInv\SPol Q$ is the smallest
$S$\hyp{}relational clone containing $Q$, i.e., the Galois closures
are exactly the $S$\hyp{}relational clones.

Section~\ref{secH} contains results about the lattice of all
$S$\hyp{}preclones on a fixed (finite) set $A$. It is shown that this
lattice is atomic and coatomic, with finitely many atoms and
coatoms. Several symmetries of the lattice, including some that are
determined by automorphisms of the monoid $S$, are established, and it
is proved that the clone lattice on $A$ is embeddable in the lattice
of $S$\hyp{}preclones in several ways. We conclude in
Section~\ref{secI} with some open problems and a brief preview of the
results that are contained in Part II.


\section{$\boldsymbol{S}$-preclones and $\boldsymbol{S}$-relational
  clones for a monoid~$\boldsymbol{S}$}\label{secC}

\begin{definition}[Operations]
  Recall that an \New{operation} on a set $A$ is a mapping
  $f \colon A^n \to A$ for some $n \in \N_{+}$. The number $n$
  (notation $\ar(f)$) is called the \New{arity} of $f$. Denote by
  $\Opa[n](A)$ the set of all $n$-ary operations on $A$, and let
  $\Op(A) := \bigcup_{n \in \N_{+}} \Opa[n](A)$.
  
  The $i$-th $n$-ary \New{projection} is the operation
  $p_i^{(n)} \in \Opa[n](A)$ given by the rule
  $p_i^{(n)}(a_1, \dots, a_n) := a_i$ for all
  $(a_1, \dots, a_n) \in A^n$.
\end{definition}

\begin{definition}[$\boldsymbol{S}$\hyp{}operations]\label{C1}
  Assume that $S$ is a fixed monoid with unit element $e$. Denote by
  $\SOpa[n](A)$ the set of all $n$-ary operations $f \colon A^n \to A$
  where each argument has a label $s \in S$. Such operations are
  called \New{$S$\hyp{}signed operations}, or
  \New{$S$\hyp{}operations} for short. Let
  $\SOp(A) := \bigcup_{n \in \N_{+}} \SOpa[n](A)$.

  The labels are given by the \New{signum} $\sgn(f)$ which is an
  $n$-tuple $\lambda = (s_{1}, \dots, s_{n}) \in S^{n}$ assigning
  $s_{i} \in S$ to the $i$-th argument of $f$
  ($i \in \{1, \dots, n\}$). We also say, that $s_{i}$ is the signum
  of the $i$-th argument of $f$. For
  $\sgn(f) = (s_{1}, \dots, s_{n})$, let
  $\Sgn(f) := \{s_{1}, \dots, s_{n}\}$ be the set of components of
  $\sgn(f)$.

  We write $f^{\lambda}$ if we want to indicate that $f \in \SOp(A)$
  has signum $\lambda$; furthermore, for $f \in \Opa[n](A)$ and
  $\lambda = (s_{1}, \dots, s_{n})$ we denote by $f^{\lambda}$ the
  same function in $\SOp(A)$ equipped with signum $\lambda$. For unary
  functions $f \in \SOpa[1](A)$ with $\sgn(f) = (s)$ we also write
  $f^{s}$ (instead of $f^{(s)}$). Sometimes it is useful to have an
  explicit notation for the \New{underlying function} of some
  $f = f^{\lambda}\in\SOp(A)$ that is obtained by just ignoring the
  signum $\lambda=\sgn(f)$. Then we shall write $\mathring f$ (or
  $f\,\mathring{}\,$); thus, $f = (\mathring f)^{\lambda}$ for
  $f \in \SOp(A)$ and $(f^{\lambda})\,\mathring{} = f$ for
  $f \in \Op(A)$.
\end{definition}

\begin{definition}[$\boldsymbol{S}$\hyp{}preclones]\label{C2}

  A set $F \subseteq \SOp(A)$ is called an \New{$S$\hyp{}preclone} if
  it contains $\id_{A}$ and is closed under the operations $\zeta$,
  $\tau$, $\nabla^s$, $\Delta$, $\circ$ that are defined as
  follows. Let $f, g \in \SOp(A)$ with
  $\sgn(f) = (s_1, \dots, s_n) \in S^n$ and
  $\sgn(g) = (s'_1, \dots, s'_m)$, and let $s \in S$. Then

  \begin{enumerate}[label=\textup{(\arabic*)}]

  \item \label{C2-1} $\sgn(\id_A) = (e)$ ($e$ is the neutral element
    of $S$) and
    \[\id_A(x):=x\quad \text{(\New{identity operation});}\]

  \item \label{C2-2} if $n \geq 2$ then
    $\sgn(\zeta f) = (s_n, s_1, \dots, s_{n-1})$ and
    \[
      (\zeta f)(x_{1}, x_{2}, \dots, x_{n}) := f(x_{2}, \dots, x_{n},
      x_{1}) \quad \text{(\New{cyclic shift});}
    \]
    if $n = 1$ then $\zeta f := f$,

  \item \label{C2-3} if $n \geq 2$ then
    $\sgn(\tau f) = (s_2, s_1, s_3, \dots, s_n)$ and
    \begin{align*}
      & (\tau f)(x_{1}, x_{2}, x_{3}, \dots, x_{n}) 
        := f(x_{2}, x_{1}, x_{3}, \dots, x_{n}) \\
      & \text{(\New{permuting the first two arguments});}
    \end{align*}
    if $n = 1$ then $\tau f := f$,

  \item \label{C2-4} for $s \in S$,
    $\sgn(\nabla^s f) = (s, s_1, \dots, s_n)$ and
    \begin{align*}
      & (\nabla^s f)(x_{1}, x_{2}, \dots, x_{n+1}) := f(x_{2}, \dots, x_{n+1}) \\
      & \text{(\New{adding a fictitious argument with signum $s$} at the first place),} 
    \end{align*} 

  \item \label{C2-5} if $n \geq 2$ and $s_{1} = s_{2} = s$, then
    $\sgn(\Delta f) = (s, s_3, \dots, s_n)$ and
    \[
      (\Delta f)(x_{1}, x_{2}, \dots, x_{n-1}) := f(x_{1}, x_{1},
      x_{2}, \dots, x_{n-1}),
    \]
    (\New{identification of the first two arguments}, if these have
    the same signum $s$); otherwise $\Delta f := f$,

  \item \label{C2-6}
    $\sgn(f \circ g)=(s'_1 s_1, \dots, s'_m s_1, s_2, \dots, s_n)$ and
    \begin{align*}
      (f \circ g)&(x_{1},\dots, x_{m},x_{m+1}, \dots, x_{m+n-1}) \\
                 &:= f(g(x_{1}, \dots, x_{m}), x_{m+1}, \dots, x_{m+n-1})
                   \qquad \text{(\New{composition}).}
    \end{align*}
  \end{enumerate}
  For $F \subseteq \SOp(A)$ the $S$\hyp{}preclone \New{generated by}
  $F$ (i.e., the least $S$\hyp{}preclone containing $F$) is denoted by
  $\SSg{F}$ or $\SSg[A]{F}$.
\end{definition}

\begin{remark}\label{C2a}

  (A) With iterated applications of the basic operations of
  Definition~\ref{C2}, we can obtain further operations on
  $S$\hyp{}operations, such as arbitrary permutations of arguments and
  respective signa, adding a fictitious argument of signum $s$ at an
  arbitrary position, identification of arguments at positions $i$ and
  $j$ if they have the same signum, or composing $f$ with $g$ in the
  $i$-th position. We can also get arbitrary compositions of the form
  \begin{multline*}
    f(g_1(x_1, \dots, x_{m_1}), g_2(x_{m_1+1}, \dots, x_{m_1+m_2}),
    \dots, \\
    g_n(x_{m_1 + \dots + m_{n-1} + 1}, \dots, x_{m_1 + \dots + m_n})),
  \end{multline*}
  where $f$ is $n$-ary and $g_i$ is $m_i$-ary for
  $i \in \{1, \dots, n\}$, and the signum of the composition is
  determined analogously to Definition~\ref{C2}\ref{C2-6}.

  (B) By adding fictitious arguments and permuting arguments, we
  obtain from $\id_{A}$ every $S$\hyp{}operation of the form
  $(p_i^{(n)})^\lambda$, where $\lambda = (s_1, \dots, s_n) \in S^n$
  with $s_i = e$. We call such $S$\hyp{}operations \New{trivial
    $S$\hyp{}operations} or \New{\textup{(}trivial\textup{)}
    projections}. All other $S$\hyp{}operations are \New{nontrivial}
  (in particular, there also exist nontrivial projections where the
  essential argument has signum $s_{i}\neq e$, e.g., the operations
  $\id^{s}:=\id_{A}^{s}:=\id_{A}^{\lambda}$ with $\lambda=(s)$,
  $s\in S\setminus\{e\}$). The trivial $S$\hyp{}operations form an
  $S$\hyp{}preclone, denoted by $\SJ_{A}$, which is the least
  $S$\hyp{}preclone and is contained in every $S$\hyp{}preclone.

  (C) In the special case when $S$ is the trivial monoid $\{e\}$, the
  labels of arguments play no essential role, and hence this case
  clearly corresponds to usual unsigned operations. In this case, the
  notion of $S$\hyp{}preclone essentially agrees with the notion of
  \New{clone}. For $F \subseteq \Op(A)$, we denote by $\Sg{F}$ or
  $\Sg[A]{F}$ the clone generated by $F$, i.e., the smallest clone
  containing $F$.

  (D) $S$\hyp{}preclones are special preclones, also known as operads,
  which can be thought of as ``clones where identification of
  arguments is not allowed'' (more precisely, they contain $\id_{A}$
  and are closed under
 $\zeta$, $\tau$, and $\circ$
 (see
  Definition~\ref{C2}\ref{C2-2}, \ref{C2-3} and \ref{C2-6}), ignoring
  all signa). The term \emph{preclone} was introduced by \'Esik and
  Weil \cite{EsiW2005} in a study of the syntactic properties of
  recognizable sets of trees. A general characterization of preclones
  as Galois closures via so-called matrix collections can be found in
  \cite{Leh2010}. The notion of \emph{operad} originates from the work
  in algebraic topology by May \cite{May1972} and Boardman and Vogt
  \cite{BoaV1973}. For general background and basic properties of
  operads, we refer the reader to the survey article by Markl
  \cite{Mar2008}.
\end{remark}

\begin{example}[``motivating example'']\label{Ex0} 
  Let $(A,\leq)$ be a poset. We consider operations $f\in\Opa[n](A)$
  ($n\in\N_{+}$) such that $f$ in each argument is either
  order\hyp{}preserving (then the argument gets signum $+$) or
  order\hyp{}reversing (signum $-$), respectively, i.e., for constants
  $c_{j}$ ($j\in\{1,\dots,n\}\setminus\{i\}$) and $x_{i},y_{i}\in A$
  ($i\in\{1,\dots,n\}$) we have
  $f(c_{1},\dots,x_{i},\dots,c_{n})\leq
  f(c_{1},\dots,y_{i},\dots,c_{n})$ whenever $x_{i}\leq y_{i}$ or
  $x_{i}\geq y_{i}$, respectively. All these operations can be seen as
  $S$\hyp{}operations and form an $S$\hyp{}preclone $F$ where
  $S := \{\mathord{+}, \mathord{-}\}$ is understood as a
  (multiplicative) group (isomorphic to the $2$\hyp{}element group
  $\{+1,-1\}$) with unit element $+$. E.g., if
  $\sgn(f)=(\mathord{-},\mathord{+})$ and
  $\sgn(g)=(\mathord{+},\mathord{-})$, then
  $(f \circ g)(x_{1},x_{2},x_{3})=f(g(x_{1},x_{2}),x_{3})$ has signum
  $(\mathord{+} \cdot \mathord{-}, \mathord{-} \cdot \mathord{-},
  \mathord{+}) = (\mathord{-}, \mathord{+}, \mathord{+})$ according to
  \ref{C2}\ref{C2-6}, which coincides with the intuition for
  composition.
\end{example}

We give two further examples in a more formalized form.
 
\begin{examples}\label{Ex1}
  Let $S':=\{\mathord{+}, o\}$ be the $2$\hyp{}element monoid with
  zero $o$ and let $\widehat{S} = \{\mathord{+}, \mathord{-}, o\}$ be
  the monoid obtained from the group $S$ from Example~\ref{Ex0} by
  adding a zero, i.e., we have the multiplication tables
  \[
    \begin{array}[c]{c|cc|}
      S' & + & o \\ \hline
      +  & + & o \\
      o  & o & o \\ \hline
    \end{array}\quad\text{ and }\quad
    \begin{array}[c]{c|ccc|}
      \widehat{S} & + & - & o \\ \hline
      +  & + & - & o \\
      -  & - & + & o \\
      o  & o & o & o \\ \hline
    \end{array}\,.
  \]

  Let $\leq$ be a partial order on $A$. Let $F'\subseteq\SpOp(A)$ and
  $\widehat{F}\subseteq\ShOp(A)$ be the set of all $S'$- or
  $\widehat{S}$-operations, respectively, such that each argument with
  signum $s\in\{\mathord{+}, \mathord{-}, o\}$ has the property as
  given in the following table:
  \begin{center}
    \begin{tabular}[t]{|cl|}
      \hline
      $s$& \quad property $P$\\\hline
      $+$& order-preserving\\
      $-$& order-reversing\\
      $o$& constant on each connected component\\\hline
    \end{tabular}
  \end{center}
  (the property for $o$ is equivalent to order\hyp{}preserving
  \textbf{and} order\hyp{}reversing). Then $F'$ as well as
  $\widehat{F}$ are $S$\hyp{}preclones for $S=S'$ and $S=\widehat{S}$,
  respectively.

  Here, for a property $P$ for unary functions $g\in A^{A}$, we define
  that an $n$-ary \New{operation $f(x_{1},\dots,x_{n})$ has property
    $P$ in an argument}, say in $x_{i}$ ($i\in\{1,\dots,n\}$), if each
  translation $x_{i}\mapsto f(c_{1},\dots,x_{i},\dots,c_{n})$ ($x_{i}$
  on the $i$-th place) has this property $P$ (with
  $c_{1},\dots,c_{i-1},c_{i+1},\dots,c_{n}\in A$).

  Note that also Example~\ref{Ex0} fits into this scheme: the
  arguments of the $S$\hyp{}operations in the $S$\hyp{}preclone $F$
  (cf.\ Example~\ref{Ex0}) have the property
  \textit{order\hyp{}preserving} if they have signum $+$, otherwise
  (signum $-$) they have the property \textit{order\hyp{}reversing}.
\end{examples}

\begin{definition}[Relations]\label{C3a}
  Recall that subsets of $A^m$ are called $m$-ary \New{relations} on
  $A$. Since $\emptyset\subseteq A^{m}$, the empty set can be
  considered as $m$-ary for arbitrary $m$. Sometimes it is convenient
  to write formally $\emptyset^{(m)}$ if $\emptyset$ is considered as
  an $m$-ary relation. Denote by $\Rela[m](A)$ the set of all $m$-ary
  relations on $A$, and let
  $\Rel(A) := \bigcup_{m \in \N_{+}} \Rela[m](A)$.

  It is often useful to think of an $m$-ary relation $\rho$ as an
  $m \times \lvert \rho \rvert$ matrix whose columns are the tuples
  belonging to $\rho$. Keeping this point of view in mind, we will
  often regard a tuple belonging to a relation as a \New{column}, and
  we will refer to its components as \New{rows}. We are shortly going
  to consider certain operations on relations, and it will be helpful
  to (informally) describe them in terms of simple manipulations of
  rows of matrices.
\end{definition}

We briefly recall the ``elementary operations'' $\zeta$, $\tau$,
$\pr$, $\times$ and $\wedge$ on relations (see, e.g.,
D. Lau~\cite[Section~II.2.3]{Lau06}).
 Let $\sigma$ and $\sigma'$ be
$m$-ary and $m'$-ary relations on a set $A$, respectively. Then
$\zeta \sigma := \sigma$, $\tau \sigma := \sigma$ and
$\pr \sigma := \sigma$ for $m = 1$, and
\begin{align*}
  \zeta \sigma &:= \{ \, (a_2, a_3, \dots, a_m, a_1) \mid (a_1, a_2, \dots, a_m) \in \sigma \, \} && (m \geq 2), \\
  \tau \sigma &:= \{ \, (a_2, a_1, a_3, \dots, a_m) \mid (a_1, a_2, \dots, a_m) \in \sigma \, \}  && (m \geq 2), \\
  \pr \sigma &:= \{ \, (a_2, \dots, a_m) \mid (a_1, a_2, \dots, a_m) \in \sigma \, \} && (m \geq 2), \\
  \sigma \times \sigma' &:= \mathrlap{\{ \, (a_1, \dots, a_m, b_1, \dots, b_{m'}) \mid 
                          \begin{array}[t]{@{}l@{}}
                            (a_1, \dots, a_m) \in \sigma, \\
                            (b_1, \dots, b_{m'}) \in \sigma' \, \},
                          \end{array}} \\
  \sigma \wedge \sigma' &:= \sigma\cap\sigma' 
                          \quad\text{(if $m\ne m'$ we put $\sigma\wedge\sigma' := \emptyset^{(m)}$).}
\end{align*}
The operation $\zeta$ is called \New{cyclic shift of rows}, $\tau$ is
called \New{transposition of the first two rows}, $\pr$ is called
\New{deletion of the first row}, $\times$ is called \New{Cartesian
  product}, and $\wedge$ is called \New{intersection}.

For $m \in \N_{+}$ and an equivalence relation $\epsilon$ on
$\{1, \dots, m\}$, let
\[
  \delta^{m}_{\epsilon}:=\delta^m_{\epsilon,A} := \{ \, (a_1, \dots,
  a_m) \in A^m \mid 
 (i, j) \in \epsilon \implies a_i = a_j \, \}.
\]
Formally we also allow $\epsilon=\top$ (where $\top$ is considered as
an extra top element in the lattice of all equivalence relations,
i.e., $\epsilon\subsetneq\top$ for all equivalence relations
$\epsilon$) and define
\[ \delta^{m}_{\top}:=\delta^{m}_{\top,A}:=\emptyset^{(m)}.\]
Relations of the form $\delta^m_{\epsilon,A}$ for some $m$ and
$\epsilon$ are called \New{diagonal relations} on $A$. The set of all
diagonal relations on $A$ is denoted by $D_{A}$. Examples of diagonal
relations are the full $m$-ary relation $A^m$, in particular
$\nabla:=\nabla_{A}:=A^{2}$ (formally this is
$\delta^{2}_{\{(1,1),(2,2)\},A}$), and the binary \New{equality}
relation $\Delta:=\Delta_{A}:= \{ \, (x,x) \mid x \in A \, \}$
(formally this is $\delta^{2}_{\{1,2\}^{2},A}$).

A set $Q$ of relations on $A$ is called a \New{relational clone} if it
is closed under the ``elementary operations'' $\zeta$, $\tau$, $\pr$,
$\times$ and $\wedge$ and contains the diagonal relations. For
$Q \subseteq \Rel(A)$, we denote by $\Rg{Q}$ or $\Rg[A]{Q}$ the
relational clone generated by $Q$, i.e., the smallest relational clone
containing $Q$.

\begin{definition}[$\boldsymbol{S}$\hyp{}relations]\label{C3}
  Let $\SRela[m](A)$ be the set of all families
  $\rho = (\rho_{s})_{s \in S}$ of $m$-ary relations
  $\rho_{s} \subseteq A^{m}$, and
  $\SRel(A) := \bigcup_{m \in \N_{+}} \SRela[m](A)$, the elements of
  which are called (finitary) \New{$S$\hyp{}relations}.

  Sometimes, instead of $\rho=(\rho_{s})_{s\in S}$ we use the notation
  $\rho=(r_{1},\dots,r_{n})$ where $r_{1},\dots,r_{n}$ is a list of
  \textsl{all} elements of all $\rho_{s}$, i.e.,
  $n=\sum_{s\in S}|\rho_{s}|$, together with the corresponding signum
  $\blambda_{\rho}:=(s_{1},\dots,s_{n})$ specifying which element
  belongs to which part $\rho_{s}$, i.e., if $r_{i}$ was chosen from
  $\rho_{s}$ then we put $s_{i}=s$, such that
  $\rho_s = \{ \, r \mid \text{$\exists i \in \{1, \dots, n\}$:
    $s_i = s$ and $r_i = r$} \, \} $. Note that in $\blambda_{\rho}$
  each $s\in S$ appears exactly $|\rho_{s}|$ times.

  Remark: The signum $\blambda_{\rho}$ is characteristic for $\rho$
  and unique up to permutation of the $s_{i}$'s. One might order the
  entries such that all signa $s_{i}$ with $s_{i}=s$ appear
  consecutively, but we keep more flexibility and allow each signum
  (as above) for $\blambda_{\rho}$ with the property that each
  $s\in S$ appears exactly $|\rho_{s}|$ times.
\end{definition}

It is straightforward to define operations on $\SRel(A)$ analogously
to the above-defined operations for usual relations, just by applying
them componentwise to $S$\hyp{}relations. This will lead to the notion
of $S$\hyp{}relational clones. We mainly adopt the approach known from
multi\hyp{}sorted algebras as developed in \cite[Section 4, pages 13--14]{LehPW2018a}.
However, we still have to introduce special, in some sense ``trivial'' $S$\hyp{}relations,
called $S$\hyp{}diagonals, which generalize the usual diagonal
relations. The motivation will become clear later in
Proposition~\ref{D1a}.

\begin{definition}[$\boldsymbol{S}$\hyp{}diagonals]\label{C4a}
  An $S$\hyp{}relation $(\rho_{s})_{s\in S}$, say $m$-ary, is called
  an \New{$S$\hyp{}diagonal} or an \New{$S$\hyp{}diagonal
    $S$\hyp{}relation}) if $\rho_{s}$ is a diagonal relation
  $\delta^{m}_{\epsilon_{s}}\in D_{A}$ for each $s\in S$ and the
  following condition is satisfied for all $s,t\in S$:
  \[
    Ss \subseteq St \implies
    \delta^{m}_{\epsilon_{s}}\subseteq\delta^{m}_{\epsilon_{t}}
  \]
  Here $St$ denotes the left ideal $\{ \, st \mid s \in S \, \}$
  generated by $t$. Note that $Ss \subseteq St \iff s \in St$ and
  $\delta^{m}_{\epsilon_{s}}\subseteq\delta^{m}_{\epsilon_{t}}\iff
  \epsilon_{t}\subseteq \epsilon_{s}$ (the case $\epsilon_{s}=\top$,
  i.e., $\delta^{m}_{\epsilon_{s}}=\emptyset$, is allowed).

  The set of all $S$\hyp{}diagonals is denoted by $\SD_{A}$.
\end{definition}

\begin{remark}\label{C4b}
  Let
  $I(S) := \{ \, s \in S \mid \exists \bar s \in S \colon \bar{s}s = e
  \, \}$ be the set of all elements for which a left inverse
  exists. Then (for finite monoids $S$) any left inverse $\bar s$ is
  also a right inverse (thus $\bar s=s^{-1}$) and $I(S)$ is a group
  (with unit $e$), the largest subgroup of $S$. We have
  $t\in I(S)\iff St=S$ (otherwise $St\subsetneq S$). If $S$ is a
  group, then an $S$\hyp{}diagonal $(\delta_{\epsilon_{s}})_{s\in S}$
  must be of the form $(\delta^{m}_{\epsilon})_{s\in S}$
  ($\epsilon_{t}=\epsilon_{s}$ for all $s,t\in S$ since $Ss=St=S$).
\end{remark}

\begin{definition}[$\boldsymbol{S}$\hyp{}relational clones]\label{C4} 
  A set of $S$\hyp{}relations is called an \New{$S$\hyp{}relational
    clone} if it contains $\delta^{S}$ and is closed under the
  operations $\zeta$, $\tau$, $\pr$, $\times$, and $\wedge$ as well as
  under the operations $\mu_{v}$ and $\sqcap^{v}$, which we will refer
  to as \New{index translation by $v$} and
  \New{$v$\hyp{}self\hyp{}intersection} for $v\in S$. These
  operations are defined as follows. For $\rho = (\rho_s)_{s \in S}$
  and $\rho' = (\rho'_s)_{s \in S}$ in $\SRel(A)$, and $v\in S$, we
  let
  \begin{enumerate}[label=\textup{(\arabic*)}]
  \item\label{C4-1} $\delta^{S} := (\Delta_{A})_{s \in S}$, i.e.,
    $\delta^{S}_{s} = \Delta_{A}$ for all $s \in S$,

  \item\label{C4-2}
    $\zeta(\rho) = \zeta ((\rho_s)_{s \in S}) := (\zeta \rho_s)_{s \in
      S}$ \; (\New{cyclic shift}),
  \item\label{C4-3}
    $\tau(\rho) = \tau ((\rho_s)_{s \in S}) := (\tau \rho_s)_{s \in
      S}$ \; (\New{transposition of the first two rows}),
  \item\label{C4-4}
    $\pr(\rho) = \pr ((\rho_s)_{s \in S}) := (\pr \rho_s)_{s \in S}$
    \;(\New{deletion of the first row}),
  \item\label{C4-5}
    $\rho \times \rho' = (\rho_s)_{s \in S} \times (\rho'_s)_{s \in S}
    := (\rho_s \times \rho'_s)_{s \in S}$ \; (\New{Cartesian
      product}),
  \item\label{C4-6}
    $\rho \wedge \rho' = (\rho_s)_{s \in S} \wedge (\rho'_s)_{s \in S}
    := (\rho_s \wedge \rho'_s)_{s \in S}$ \; (\New{intersection}),

  \item\label{C4-7}
    $\mu_v(\rho) = \mu_v((\rho_{s})_{s \in S}) := (\rho_{sv})_{s \in
      S}$ \; (\New{index translation by $v$}),

  \item\label{C4-8}
    $\sqcap^{v}\rho =((\sqcap^{v}\rho)_{s})_{s\in S}
    :=(\bigcap\{\rho_{s'}\mid s'v=s\})_{s\in S}$
    (\New{$v$-self-intersection})\\
    {\small(i.e., via the (right)
      multiplicative action of an element $v\in S$)},\\
    in particular, $(\sqcap^{v}\rho)_{s}=\bigcap\emptyset=A^{m}$ if
    $s\in S\setminus Sv$ and $\ar(\rho)=m$.

  \end{enumerate}

  For $Q \subseteq \SRel(A)$ the \New{$S$\hyp{}relational clone
    generated by $Q$}, i.e., the least $S$\hyp{}relational clone
  containing $Q$, is denoted by $\SRg{Q}$ or $\SRg[A]{Q}$.
\end{definition}

The operations \ref{C4-7} and \ref{C4-8} are special cases of a more
general operation called $\cM$-self-intersection:
\begin{definition}
  \label{C4-Self} Let $\cM=(M_{s})_{s\in S}$ be a family of subsets of
  $S$ satisfying the following condition
  \begin{align*}
    \text{(*)}&\quad \forall s,s'\in S:
                s'M_{s}\subseteq M_{s's}\,.
  \end{align*}
  Then we define
  \begin{enumerate}[label=\textup{(\arabic*)},ref=\textup{(\arabic*)}]\addtocounter{enumi}{8}
  \item\label{C4-9} \label{C4-7y}
    $\sqcap_{\cM}\rho=((\sqcap_{\cM}\rho)_{s})_{s\in S} := (\bigcap \{
    \, \rho_{s'} \mid s' \in M_{s} \, \})_{s\in S}$ \quad
    \text{(\New{$\cM$\hyp{}self\hyp{}intersection}),}
  \end{enumerate}
  in particular, $(\sqcap_{\cM}\rho)_{s}:=A^{m}$ if $M_{s}=\emptyset$
  and $\ar(\rho)=m$.
    
  Note that $(\sqcap_{\cM}\rho)_{s}\subseteq\rho_{s'}$ for all
  $s'\in M_{s}$.
\end{definition}

\begin{remark}\label{LM} For $v\in S$, consider the special families 
  $\cT^{v}:=(\{sv\})_{s\in S}$ and $\cM^{v}=(M^{v}_{s})_{s\in S}$ with
  $M^{v}_{s} := \{ x \in S \mid xv = s \}$ (which easily can be seen
  to satisfy condition (*) of Definition~\ref{C4-Self}). Then we have
  $\mu_{v}(\rho)=\sqcap_{\cT^{v}}(\rho)$ (cf.~\ref{C4}\ref{C4-7}) and
  $\sqcap^{v}\rho=\sqcap_{\cM^{v}}\rho$ (cf.~\ref{C4}\ref{C4-8}).

  For $v \in S$ let $\alpha_{v}\colon S \to S$, $x \mapsto xv$ be the
  right multiplication with $v$. Then $M^{v}_{s}=\alpha_{v}^{-1}(s)$
  for $s\in S$. For later use we mention the following properties of
  $\cM^{v}=(M^{v}_{s})_{s\in S}$:
  \begin{enumerate}[label=\textup{(\alph*)}]
  \item\label{LMa} $e\in M^{v}_{v}$,

  \item\label{LMb} $M^{v}_{s}\cap M^{v}_{t}=\emptyset$ for $s,t\in S$,
    $s\neq t$,

  \item\label{LMc} $\forall s,t\in S: s\in M^{v}_{t}\implies sv=t$.
  \end{enumerate}
  These properties can be checked easily (e.g., \ref{LMb} follows from
  the fact that the $M^{v}_{s}$ ($s\in S$) are the equivalence classes
  of the kernel of $\alpha_{v}$).
\end{remark}

\begin{remark}\label{C5} 
  (A) If $v\in S$ is an invertible element of the monoid $S$, then
  $M^{v}_{s}=\{sv^{-1}\}$ for $s\in S$, thus $\cM^{v}=\cT^{v^{-1}}$,
  in particular the operation $\sqcap^{v}$ equals $\mu_{v^{-1}}$
  (cf.~\ref{LM} and \ref{C4}\ref{C4-7},\ref{C4-8}). Consequently, if
  $S$ is a group and therefore each element is invertible,
  $S$\hyp{}relational clones are characterized by the closure under
  the operations \ref{C4-1}--\ref{C4-7} of Definition~\ref{C4}.

  (B) Analogously to the case of usual relational clones
  (cf.~\cite[Part~II, 2.5]{Lau06} or \cite[1.1.9]{PoeK79}), with the
  operations \ref{C4-1}--\ref{C4-8} it is possible to construct many
  other operations under which an $S$\hyp{}relational clone is
  closed. We mention here some: all $S$\hyp{}diagonal
  $S$\hyp{}relations (see Lemma~\ref{C4c}), arbitrary permutation of
  rows, deletion of arbitrary rows, identification of rows, doubling
  of rows, relational product, and so on. In particular, we will need
  in the proofs of later theorems the following operations that are
  derivable from the ``elementary operations''. For an $m$-ary
  relation $\rho$ and any $z_1, \dots, z_t \in \{1, \dots, m\}$ (not
  necessarily distinct elements), the \New{projection} of $\rho$ to
  rows $z_1, \dots, z_t$ is the $t$-ary relation
  \[
    \pr_{z_1, \dots, z_t}(\rho) :=
 \{ \, (a_{z_1}, \dots, a_{z_t})
    \in A^t \mid (a_1, \dots, a_m) \in \rho \, \}.
  \]
  This naturally extends to $S$\hyp{}relations by componentwise
  application: for $\rho = (\rho_s)_{s \in S} \in \SRela[m](A)$,
  $\pr_{z_1, \dots, z_t}(\rho) := (\pr_{z_1, \dots, z_t}(\rho_s))_{s
    \in S}$. Clearly, $\pr(\rho)=\pr_{2,\dots,m}(\rho)$ for an $m$-ary
  $\rho$ (cf.\ Definition~\ref{C4}\ref{C4-4}).

  Note also that the \New{empty $S$\hyp{}relation}
  $(\emptyset)_{s \in S}$ (that can be considered as
  $(\emptyset^{(m)})_{s\in S}$ for an arbitrary $m\in\N_{+}$) belongs
  to every $S$\hyp{}relational clone; it is obtained by taking the
  intersection of two $S$\hyp{}relations of distinct arities.

  (C) It is well known (see, e.g., \cite[2.1.3(i)]{PoeK79}) that in
  the classical case of usual (unsigned) relational clones, the
  relational clone $\Rg{Q}$ generated by a set $Q$ of relations (using
  the ``elementary operations'' $\delta,\zeta,\tau,\pr,\times,\land$)
  equals the set of relations that are primitively positively
  first-order definable (pp\hyp{}definable) from the relations in
  $Q$. More explicitly (cf.~\cite[1.6]{Poe04a}): Each primitive
  positive first order formula
  $\phi=\phi(\bar\sigma_{1},\dots,\bar\sigma_{q};x_{1},\dots,x_{m})$,
  (i.e., $\phi$ contains only $\exists, \land, =$, and relation
  symbols, say $m_{i}$-ary $\bar\sigma_{i}$, $i\in \{1,\dots,q\}$, and
  free variable symbols, say $x_{1},\dots,x_{m}$) defines a so-called
  \New{logical operation} $t_{\phi}$ which can be applied to
  $m_{i}$-ary relations $\sigma_{i}\subseteq A^{m_{i}}$ and gives the
  $m$-ary relation
  \[
    t_{\phi}(\sigma_{1},\dots,\sigma_{q}) := \{ \,
    (a_{1},\dots,a_{m})\in A^{m} \mid {} \models
    \phi(\sigma_{1},\dots,\sigma_{q};a_{1},\dots,a_{m}) \, \}.
  \]
  Then the closure $\Rg{Q}$ is the closure under all
  pp\hyp{}definitions, which means explicitly the closure under all
  logical operations $t_{\phi}$.

  As for $S$\hyp{}relational clones, it follows from
  Definition~\ref{C4} that \emph{a set of $S$\hyp{}relations is an
    $S$\hyp{}relational clone if and only if it is closed under
    pp\hyp{}definitions} (now we apply the same pp\hyp{}formula to
  each component of $S$\hyp{}relations), \emph{index translations} and
  \emph{self\hyp{}intersections}. Moreover, \emph{$S$-relational
    clones are closed also under $\cM$-self-intersections} for
  arbitrary families $\cM$ satisfying \ref{C4-Self}(*). In order to
  see this, in the proof of Lemma~\ref{D3A} we shall explicitly show
  that each set of the form $\SInv F$ -- and therefore, due to
  Theorem~\ref{Thm-II}, also each $S$-relational clone -- is closed
  under $\cM$-self-intersections.
\end{remark}

\begin{lemma}\label{C4c} 
  $\SD_{A}=\SRg[A]{\emptyset}=\SRg[A]{\delta^{S}}$ is the least
  $S$\hyp{}relational clone contained in every $S$\hyp{}relational
  clone.
\end{lemma}
\begin{proof}
  We show that each $S$\hyp{}diagonal $S$\hyp{}relation can be
  generated from $\delta^{S}$. It is well known that
  $D_{A}=\Rg[A]{\Delta}$ (cf., e.g., \cite[1.1.9(R1)]{PoeK79}, the
  diagonal relation $\delta_{3}^{\{1;2,3\}}(A)$ used there, can be
  derived from $\Delta$ with the ``elementary operations'' from
  Definition~\ref{C3a} as follows:
  $\delta_{3}^{\{1;2,3\}}(A)=\pr(\Delta)\times \Delta$). Thus all
  $(\delta^{m}_{\epsilon})_{s\in S}$ can be constructed from
  $\delta^{S}=(\Delta)_{s\in S}$. We show that all further
  $S$\hyp{}diagonals can be derived from these.

  Let $\rho=(\delta^{m}_{\epsilon_{s}})_{s\in S}\in \SD_{A}$ (for
  notation, see Definition~\ref{C4a}) and
  $\delta^{S}_{\epsilon}:=(\delta^{m}_{\epsilon})_{s\in S}$. According
  to Definition~\ref{C4}\ref{C4-8}, for $v\in S$,
  $(\sqcap^{v}\delta^{S}_{\epsilon})_{s}$ equals
  $\delta^{m}_{\epsilon}$ if $s\in Sv$ and it equals $A^{m}$ if
  $s\in S\setminus Sv$. Consequently, the conditions in \ref{C4a} for
  the $S$-diagonal $\rho$ imply
  $\rho=\bigwedge \{ \, \sqcap^{v}\delta^{S}_{\epsilon_{v}} \mid v \in
  S \, \} \in \SRg{\delta^{S}}$.

  It remains to mention that $\SD_{A}$ really is an
  $S$\hyp{}relational clone. Namely, the closure under each of the
  operations \ref{C4-1}--\ref{C4-8} of Definition~\ref{C4}
 can be
  directly checked (however, it also follows immediately from
  Proposition~\ref{D1a} and Lemma~\ref{D3A}).
\end{proof}


\section{Invariant $\boldsymbol{S}$-relations and
  $\boldsymbol S$-polymorphisms \linebreak and the Galois connection
  $\boldsymbol{\SPol}$--$\boldsymbol{\SInv}$}\label{secD}

In this section we introduce the Galois connection $\SPol$--$\SInv$
induced by the $S$\hyp{}preservation relation $\Spreserves$ and give
some preliminary results. This parallels the classical Galois
connection $\Pol$--$\Inv$ induced by the preservation relation on
usual (unsigned) operations and relations, which we will first briefly
recall.

\begin{notation}
  Let $f \in \Opa[n](A)$, and let
  $r_i = (r_i^{(1)}, \dots, r_i^{(m)}) \in A^m$
  ($i \in \{1, \dots, n\}$) We write
  \[
    f(r_1, \dots, r_n) := (f(r_1^{(1)}, \dots, r_n^{(1)}), \dots,
    f(r_1^{(m)}, \dots, r_n^{(m)}) )
  \]
  (componentwise application of $f$ to $m$-tuples). Furthermore, if
  $\rho_i \subseteq A^m$ ($i \in \{1, \dots, n\}$), then we let
  \[
    f( \rho_1, \dots, \rho_n ) :=
 \{ \, f(r_1, \dots, r_n) \mid r_1
    \in \rho_1, \dots, r_n \in \rho_n \, \}.
  \]
\end{notation}

\begin{definition}[Preservation]
  Let $f \in \Opa[n](A)$ and $\rho \in \Rela[m](A)$. We say that $f$
  \New{preserves} $\rho$ (or $f$ is a \New{polymorphism} of $\rho$, or
  $\rho$ is an \New{invariant} of $f$), and we write
  $f \preserves \rho$, if $f(\rho, \dots, \rho) \subseteq \rho$.
\end{definition}

The relation $\preserves$ induces a Galois connection between
operations and relations. The corresponding operators are denoted as
follows.

\begin{definition}\label{def:Pol-Inv}
  Let $F \subseteq \Op(A)$ and $Q \subseteq \Rel(A)$. Then we define
  \begin{align*}
    \Pol Q &:= \{ \, f \in \Op(A) \mid \forall\, \rho \in Q \colon f \preserves \rho \, \}
    && \text{(polymorphisms)}, \\
    \Inv F &:= \{ \, \rho \in \Rel(A) \mid \forall\, f \in F \colon f \preserves \rho \, \}
    && \text{(invariant relations)}.
  \end{align*}
\end{definition}

It is well known that $\Pol Q$ is a clone for any
$Q \subseteq \Rel(A)$ and $\Inv F$ is a relational clone for any
$F \subseteq \Op(A)$. Moreover, if $A$ is finite, then it holds that
$\Sg{F} = \Pol \Inv F$ and $\Rg{Q} = \Inv \Pol Q$ (\cite{BodKKR69a},
cf.\ \cite[Folgerung 1.2.4]{PoeK79}).

\begin{definition}[$\boldsymbol{S}$\hyp{}preservation]\label{D1}
  Let $f \in \SOp(A)$ with $\sgn(f) = (s_{1}, \dots, s_{n})$ and
  $\rho = (\rho_{s})_{s \in S} \in \SRela[m](A)$. The preservation
  property $\Spreserves$ is defined by
  \begin{align}\tag{1}\label{D1i}
    f \Spreserves (\rho_{s})_{s \in S} :\iff \forall s \in S \colon
    f(\rho_{s_{1}s}, \dots, \rho_{s_{n}s}) \subseteq \rho_{s}.
  \end{align}
  If $f \Spreserves \rho$, then we say that $f$
  \New{$S$\hyp{}preserves} $\rho$, or $f$ is an
  \New{$S$\hyp{}polymorphism} of $\rho$, or $\rho$ is \New{invariant}
  for $f$. Note that in particular we have
  $f(\rho_{s_{1}}, \dots, \rho_{s_{n}}) \subseteq \rho_{e}$ ($e$ is
  the neutral element of $S$). If it is clear from the context that we
  deal with $S$\hyp{}operations and $S$\hyp{}relations, then, for
  $f\Spreserves\rho$, we also say $f$ \New{preserves} $\rho$ or $f$ is
  a \New{polymorphism} of $\rho$.
\end{definition}

For $S$\hyp{}relations or $S$\hyp{}operations of a special form, the
$S$\hyp{}preservation property can be expressed by the usual
preservation property:

\begin{lemma}\label{D2}
  \leavevmode
  \begin{enumerate}[label=\textup{(\roman*)}]
  \item\label{D2a} For $f \in \SOp(A)$ and $\sigma \in \Rel(A)$ we
    have
    \begin{align*}
      f \Spreserves (\sigma)_{s \in S} \iff \mathring{f} \preserves \sigma.
    \end{align*}
    Here $(\sigma)_{s \in S}$ is the $S$\hyp{}relation
    $(\rho_{s})_{s \in S}$ with $\rho_{s} = \sigma$ for all $s \in S$.

  \item \label{D2b} If $\sgn(f) = (e, \dots, e)$ and
    $\rho = (\rho_{s})_{s \in S}$, then
    \begin{equation*}
      f \Spreserves \rho
      \iff
      \forall s \in S \colon \mathring{f} \preserves \rho_{s}.
    \end{equation*}
  \end{enumerate}
\end{lemma}

\begin{proof}
  This follows immediately from the definition of
  $S$\hyp{}preservation (for \ref{D2b} see also Proposition~\ref{D8}).
\end{proof}

The $S$\hyp{}preservation relation $\Spreserves$ induces a Galois
connection between $S$\hyp{}operations $\SOp(A)$ and
$S$\hyp{}relations $\SRel(A)$. The corresponding operators are denoted
as follows.

\begin{definition}\label{D3}
  Let $F \subseteq \SOp(A)$ and $Q \subseteq \SRel(A)$. Then we define
  \begin{align*}
    \SPol Q &:= \{ \, f \in \SOp(A) \mid \forall\, \rho \in Q \colon f \Spreserves \rho \, \} && \text{($S$\hyp{}polymorphisms)}, \\
    \SInv F &:= \{ \, \rho \in \SRel(A) \mid \forall\, f \in F \colon f \Spreserves \rho \, \} && \text{(invariant $S$\hyp{}relations)}.
  \end{align*}
\end{definition}

\begin{examples}\label{Ex0a}
  Let $(A,\leq)$ be a poset and $S = \{\mathord{+},
  \mathord{-}\}$. Then the $S$\hyp{}preclone $F$ considered in
  Example~\ref{Ex0} can be characterized as
  \[
    F = \SPol \rho \text{ for the $S$\hyp{}relation }
    \rho=(\rho_{+},\rho_{-}) := (\mathord{\leq}, \mathord{\geq}).
  \]

  Moreover, let
  $\sigma:=(B_{1}\times B_{1})\cup\ldots\cup(B_{m}\times B_{m})$ where
  $B_{1},\dots,B_{m}$ are the connected components of the partial
  order $\leq$. Then the $S'$-preclone $F'$ and the
  $\widehat{S}$-preclone $\widehat{F}$ considered in Example~\ref{Ex1}
  can be characterized by $S'$- and $\widehat{S}$-relations, resp., as
  follows:
  \begin{align*}
    F' &= \SpPol \{(\mathord{\leq}, \Delta_{A}), (\Delta_{A}, \sigma)\} \,, \\
    \widehat{F} &= \ShPol \{(\mathord{\leq}, \mathord{\geq}, \Delta_{A}), (\Delta_{A}, \Delta_{A}, \sigma)\} \,.
  \end{align*}
\end{examples}

Trivial projections and $S$\hyp{}diagonals play a special role: they
can be considered as trivial with respect to $S$\hyp{}preservation (in
particular, this motivates why $S$\hyp{}diagonals should belong to
each $S$\hyp{}relational clone, cf.\ Definition~\ref{C4}\ref{C4-1} and
Lemma~\ref{C4c}):

\begin{proposition}\label{D1a}
  We have $\SJ_{A}=\SPol\SRel(A)$ and $\SD_{A}=\SInv\SOp(A)$, i.e.,
  \textup{(}trivial\textup{)} projections are those
  $S$\hyp{}operations which $S$\hyp{}preserve \textsl{every}
  $S$\hyp{}relation, and $S$\hyp{}diagonals are those
  $S$\hyp{}relations which are invariant for \textsl{every}
  $S$\hyp{}operation.
\end{proposition}

\begin{proof}
  First we prove $\SJ_{A}=\SPol\SRel(A)$. It is easy to check that any
  (trivial) projection (see Remark~\ref{C2a}(B)) preserves every
  relation, i.e., $\SJ_{A}\subseteq\SPol\SRel(A)$. Conversely, let
  $f\in\SPol\SRel(A)$ with $\sgn(f)=(s_{1},\dots,s_{n})$. From
  Lemma~\ref{D2}\ref{D2a} we conclude that
  $\mathring{f}\preserves\sigma$ for each $\sigma\in\Rel(A)$. Because
  $\Pol\Rel(A)=J_{A}$ (this follows from well\hyp{}known results about
  clones: since $\Inv J_{A}=\Rel(A)$, see \cite[1.1.15]{PoeK79}, one
  can conclude $J_{A}=\Pol\Inv J_{A}=\Pol\Rel(A)$) we get
  $\mathring{f}\in J_{A}$ is a projection, i.e.,
  $f(x_{1},\dots,x_{n})=x_{i}$. If $s_{i}\ne e$ then $f$ does not
  preserve $\rho:=(\Delta,\dots,\Delta,\nabla,\Delta,\dots,\Delta)$
  (with $\nabla$ at the $i$-th place), a contradiction to the choice
  of $f$. Thus $s_{i}=e$ and $f\in \SJ_{A}$; consequently
  $\SPol\Rel(A)\subseteq \SJ_{A}$.

  Now we show $\SD_{A}=\SInv\SOp(A)$. Let $f\in \SOp(A)$ with
  $\sgn(f)=(s_{1},\dots,s_{n})$ and
  $\rho=(\rho_{s})_{s\in S}=(\delta^{m}_{\epsilon_{s}})_{s\in S}\in
  \SD_{A}$, in particular we have $\rho_{s_{i}t}\subseteq\rho_{t}$ for
  all $t \in S$ since $s_{i}t\in St$ (cf.\ Definition~\ref{C4a}) for
  $i\in\{1,\dots,n\}$. Thus
  $f(\rho_{s_{1}t},\dots,\rho_{s_{n}t})\subseteq
  f(\rho_{t},\dots,\rho_{t})\subseteq\rho_{t}$ (note for the last
  inclusion that the diagonal $\rho_{t}\in D_{A}$ is preserved by
  every function, thus also by $\mathring f$), which shows that
  $f\Spreserves\rho$. Thus $\SD_{A}\subseteq\SInv\SOp(A)$.

  Conversely, let $\rho\in\SInv\SOp(A)$. Note that
  $\{ \, \mathring f \mid f \in \SOp(A), \; \Sgn(f) = \{e\} \, \} =
  \Op(A)$.
 From Lemma~\ref{D2}\ref{D2b} we conclude that
  $\mathring{f} \preserves \{ \, \rho_{s} \mid s \in S \, \}$ for
  every $\mathring{f} \in \Op(A)$. Thus
  $\{ \, \rho_{s} \mid s \in S \, \} \subseteq \Inv\Op(A) = D_{A}$
  (the latter equality is well known, see e.g.,
  \cite[1.1.15]{PoeK79}). Furthermore, let $Ss\subseteq St$, i.e.,
  there exists $s_{1}\in S$ such that $s=s_{1}t$. Consequently, from
  $\id^{s_{1}}\Spreserves\rho$ we get
  $\rho_{s}=\id^{s_{1}}(\rho_{s})=\id^{s_{1}}(\rho_{s_{1}t})\subseteq\rho_{t}$,
  which proves $\rho\in\SD_{A}$ (because the compatibility condition
  in Definition~\ref{C4a} is fulfilled). Thus we also have
  $\SInv\SOp(A)\subseteq\SD_{A}$.
\end{proof}

The operators $\SPol$ and $\SInv$ of the Galois connection produce
$S$\hyp{}preclones and $S$\hyp{}relational clones, as shown by the
following lemma.
 The main question, if \textsl{every}
$S$\hyp{}preclone and $S$\hyp{}relational clone can be ``produced'' in
this way, shall be answered in the next section (Theorems~\ref{Thm-I}
and \ref{Thm-II}).

\begin{lemma}\label{D3A}
  Let $F \subseteq \SOp(A)$ and $Q \subseteq \SRel(A)$. Then $\SPol Q$
  is an $S$\hyp{}preclone and $\SInv F$ is an $S$\hyp{}relational
  clone. Moreover, $\SPol Q = \SPol \SRg{Q}$ and
  $\SInv F = \SInv \SSg{F}$.
\end{lemma}

\begin{proof}
  We show first that $\SPol Q$ is an $S$\hyp{}preclone. Let
  $\rho = (\rho_s)_{s \in S}$ be an $S$\hyp{}relation in $Q$. For all
  $s \in S$, $\id_{A}(\rho_{es}) = \rho_{es} = \rho_s$; hence
  $\id_{A} \in \SPol Q$. Let $f, g \in \SPol Q$ with
  $\sgn(f) = (s_1, \dots, s_m) \in S^m$ and
  $\sgn(g) = (s'_1, \dots, s'_n) \in S^n$, and let $s \in S$. If
  $m = 1$, then we have $\zeta f = \tau f = f \in \SPol Q$; otherwise
  \begin{align*}
    & (\zeta f)(\rho_{s_n s}, \rho_{s_1 s}, \dots \rho_{s_m s}) = f(\rho_{s_1 s}, \dots \rho_{s_m s}) \subseteq \rho_s, \\
    & (\tau f)(\rho_{s_2 s}, \rho_{s_1 s}, \rho_{s_3 s}, \dots, \rho_{s_m s}) = f(\rho_{s_1 s}, \dots \rho_{s_m s}) \subseteq \rho_s. \\
    \intertext{For each $t \in S$,}
    & (\nabla^t f)(\rho_{t s}, \rho_{s_1 s}, \dots \rho_{s_m s}) = f(\rho_{s_1 s}, \dots \rho_{s_m s}) \subseteq \rho_s, \\
    \intertext{and if $s_1 = s_2 = t$ then}
    & (\Delta f)(\rho_{s_1 s}, \rho_{s_3 s}, \dots \rho_{s_m s}) \subseteq f(\rho_{s_1 s}, \dots \rho_{s_m s}) \subseteq \rho_s; \\
    \intertext{otherwise $\Delta f = f \in \SPol Q$.
    Finally,}
    & (f \circ g)(\rho_{s'_1 s_1 s}, \dots, \rho_{s'_n s_1 s}, \rho_{s_2 s}, \dots, \rho_{s_m s})
    \\ & \quad = f(g(\rho_{s'_1 s_1 s}, \dots, \rho_{s'_n s_1 s}), \rho_{s_2 s}, \dots, \rho_{s_m s})
    \\ & \quad \subseteq f(\rho_{s_1 s}, \rho_{s_2 s}, \dots, \rho_{s_m s})
         \subseteq \rho_s.
  \end{align*}
  This shows that $\SPol Q$ is an $S$\hyp{}preclone.

  Now we show that $\SInv F$ is an $S$\hyp{}relational clone. By
  Proposition~\ref{D1a} we have
  $\delta^{S}\in \SD_{A}\subseteq \SInv F$, thus condition \ref{C4-1}
  of Definition~\ref{C4} is satisfied. Let $\rho, \rho' \in \SInv F$,
  say $\rho$ is $m$-ary and $\rho'$ is $m'$-ary. If $m = 1$, then
  $\zeta \rho = \tau \rho = \pr \rho = \rho \in \SInv F$. If
  $m \neq m'$, then
  $\rho \wedge \rho' = (\emptyset^{(m)})_{s \in S} \in \SInv F$. In
  all other cases, it is straightforward to verify that for $f\in F$
  with $\sgn(f)=(s_{1},\dots,s_{n})$ and $\pi \in \{\zeta,\tau, \pr\}$
  we have
  \begin{align*}
    &
      f(\pi \rho_{s_1 s}, \dots, \pi \rho_{s_n s}) =
      \pi f(\rho_{s_1 s}, \dots, \rho_{s_n s})
      \subseteq \pi \rho_s,
    \\ &
         f(\rho_{s_1 s} \times \rho'_{s_1 s}, \dots, \rho_{s_n s}
         \times  \rho'_{s_n s})
    \\ & \quad
         \subseteq
         f(\rho_{s_1 s}, \dots, \rho_{s_n s}) \times f(\rho'_{s_1 s},
         \dots, \rho'_{s_n s}) 
         \subseteq
         \rho_s \times \rho'_s,
    \\ &
         f(\rho_{s_1 s} \land \rho'_{s_1 s}, \dots, \rho_{s_n s}
         \land  \rho'_{s_n s})
    \\ & \quad
         \subseteq
         f(\rho_{s_1 s}, \dots, \rho_{s_n s}) \cap f(\rho'_{s_1 s},
         \dots, \rho'_{s_n s}) 
         \subseteq
         \rho_s \land \rho'_s,
  \end{align*}
  so $f \Spreserves \pi \rho$, $f \Spreserves \rho \times\rho'$ and
  $f \Spreserves \rho \land\rho'$.

  For the remaining two operations (i.e., index translation and
  self\hyp{}intersection, see
  Definition~\ref{C4}\ref{C4-7},\ref{C4-8}) we use the more general
  operation $\cM$-self-intersection (see
  Defininition~\ref{C4-Self}\ref{C4-9}). Because of \ref{LM} this will
  prove the result for the operations \ref{C4-7} and \ref{C4-8}, too.

  Thus let $f\in F$ with $\sgn(f)=(s_{1},\dots,s_{n})$ and
  $\rho\in\SInv F$, i.e., $f\Spreserves\rho$. Furthermore, let
  $\cM=(M_{s})_{s\in S}$ satisfy the condition (*) in
  Definition~\ref{C4}\ref{C4-9}. For shorter notation put
  $\rho':=\sqcap_{\cM}\rho$. We have to show
  $ f(\rho'_{s_{1}s},\dots,\rho'_{s_{n}s})\subseteq\rho'_{s}$ for
  every $s\in S$. Since $s_{i}s'\in s_{i}M_{s}\subseteq M_{s_{i}s}$
  for each $s'\in M_{s}$, by the definition of $\rho'_{s_{i}s}$ we
  have $\rho'_{s_{i}s}\subseteq \rho_{s_{i}s'}$ for all $s'\in M_{s}$
  ($i\in\{1,\dots,n\}$). Consequently,
  \[
    f(\rho'_{s_{1}s},\dots,\rho'_{s_{n}s})\subseteq
    f(\rho_{s_{1}s'},\dots,\rho_{s_{n}s'})\subseteq \rho_{s'}
  \]
  for each $s'\in M_{s}$, thus
  $ f(\rho'_{s_{1}s},\dots,\rho'_{s_{n}s})$ is contained in
  $\rho'_{s}$ being the intersection of all such
  $\rho_{s'}$. Consequently, $f\Spreserves\sqcap_{\cM}\rho$.

  This shows that $\SInv F$ is an $S$\hyp{}relational clone.

  Finally, the last statement of the Lemma follows from the
  above. Indeed, we clearly have $\SPol \SRg{Q} \subseteq \SPol Q$ and
  $Q \subseteq \SInv \SPol Q$ (by the general properties of a Galois
  connection, here $\SPol$--$\SInv$). The latter implies
  $\SRg{Q} \subseteq \SInv \SPol Q$ since $\SInv \SPol Q$ is an
  $S$\hyp{}relational clone (by what we have shown above);
  consequently
  $\SPol Q = \SPol \SInv \SPol Q \subseteq \SPol \SRg{Q}$, which
  proves the equality $\SPol Q = \SPol \SRg{Q}$. The equality
  $\SInv F = \SInv \SSg{F}$ follows analogously from the fact that
  $\SPol \SInv F$ is an $S$\hyp{}preclone.
\end{proof}

A characterization of the (usual) clones generated by
$S$\hyp{}operations, in particular by $S$\hyp{}preclones $F$, is also
possible.

\begin{proposition}\label{D4}
  Let $F \subseteq \SOp(A)$. Then the clone
  $\Sg{\mathring{F}} \subseteq \Op(A)$ generated by $\mathring{F}$
  \textup{(}i.e., all functions of $F$ ignoring the signum of the
  operations\textup{)} can be characterized as follows:
  \[
    \Sg{\mathring{F}} = \Pol \{ \, \sigma \mid (\sigma)_{s \in S} \in
    \SInv F \, \}.
  \]
\end{proposition}

\begin{proof}
  It is well known that $\Sg{\mathring{F}} = \Pol \Inv \mathring{F}$
  (\cite{BodKKR69a}, cf.\ \cite[Folgerung 1.2.4]{PoeK79}). By
  Lemma~\ref{D2}\ref{D2a} we have
  $\sigma \in \Inv \mathring{F} \iff (\sigma)_{s \in S} \in \SInv F$,
  which finishes the proof.
\end{proof}

\begin{definition}\label{D5}
  Let $F \subseteq \SOp(A)$ and $\rho \in \SRela[m](A)$. Let
  \[
    \Gamma_{F}(\rho) := \bigcap \{ \, \sigma \in \SRela[m](A) \mid
    \rho \subseteq \sigma \in \SInv F \, \},
  \]
  where $\rho \subseteq \sigma$ means $\rho_{s} \subseteq \sigma_{s}$
  for each $s \in S$. Thus $\Gamma_{F}(\rho)$ is the least
  $S$\hyp{}relation which contains $\rho$ and is invariant for $F$
  (note that $\SInv F$ is closed under intersections, cf.\
  Definition~\ref{C4}). Moreover, we have
  $\Gamma_{F}(\rho) = \Gamma_{\SSg{F}}(\rho)$ (by definition, since
  $\SInv F = \SInv \SSg{F}$; see Lemma~\ref{D3A}).
\end{definition}

\begin{lemma}[{cf.\ \cite[1.1.20]{PoeK79}}]\label{LemGammaF}
  For any $F\subseteq\SOp(A)$ and $\rho \in \SRel(A)$,
  \begin{gather*}
    \zeta \Gamma_F(\rho) = \Gamma_F(\zeta \rho),
 \qquad
 \tau
    \Gamma_F(\rho) = \Gamma_F(\tau \rho),
    \\
    \pr_{z_1, \dots, z_m}(\Gamma_F(\rho)) = \Gamma_F(\pr_{z_1, \dots,
      z_m}(\rho)).
  \end{gather*}
\end{lemma}

\begin{proof}
  The claims about $\zeta$ and $\tau$ are easy to verify. For the
  claim about $\pr_{z_1, \dots, z_m}$, observe that
  $\pr_{z_1, \dots, z_m}(\rho) \subseteq \pr_{z_1, \dots,
    z_m}(\Gamma_F(\rho)) \in \SInv F$ and that every $S$\hyp{}relation
  $\sigma \in \SInv F$ that contains $\pr_{z_1, \dots, z_m}(\rho)$ can
  be turned into a relation $\sigma' \in \SInv F$ with
  $\rho \subseteq \sigma'$ by introduction of fictitious rows and
  identification of equal rows.
\end{proof}

\begin{proposition}[Characterization of $\boldsymbol{\Gamma_F(\rho)}$]\label{D5a}
  \leavevmode
  \begin{enumerate}[label=\textup{(\roman*)}]
  \item \label{D5a-i}

    Let $F\subseteq\SOp(A)$ and $\rho\in\SRel(A)$. For $i\in\N$ we
    define
    \begin{align*}
      \rho^{(0)}&:=\rho\\
      \rho^{(i+1)}&:=(\rho^{(i+1)}_{s})_{s\in S}, \text{ where}\\
      \rho^{(i+1)}_{s}&:=\rho^{(i)}_{s}\cup \{ \, f(r_{1},\dots,r_{n}) \mid
                        \begin{array}[t]{@{}l@{}}
                          f \in F^{(n)}, \, n \in \N_{+},  \\
                          r_{1}\in\rho^{(i)}_{s_{1}s}, \dots, r_{n}\in\rho^{(i)}_{s_{n}s}, \\
                          \text{\textup{where }}(s_{1},\dots,s_{n}) := \sgn(f) \, \}.
                        \end{array}
    \end{align*}
    Then we have $\Gamma_{F}(\rho)=\bigcup_{i=0}^{\infty}\rho^{(i)}$.

  \item\label{D5a-ii} Let $r\in\Gamma_{F}(\rho)_{s}$ for some
    $s\in S$. Then there exist an $S$\hyp{}operation $f\in\SSg{F}$
    with $\sgn(f)=(s_{1},\dots,s_{q})$ \textup{(}for some $q\in\N_{+}$
    and $s_{1},\dots,s_{q}\in S$\textup{)} and $r_{j}\in\rho_{s_{j}s}$
    \textup{(}$j\in\{1,\dots,q\}$\textup{)} such that
    $r=f(r_{1},\dots,r_{q})$.

  \item\label{D5a-iii} Let $\rho\in\SRel(A)$ and
    $\blambda_{\rho}=(s_{1},\dots,s_{n})$, $\rho=(r_{1},\dots,r_{n})$
    \textup{(}cf.\ Definition~\ref{C3}\textup{)}. Then
    \[
      (\Gamma_{F}(\rho))_{e} = \{ \, f(r_{1},\dots,r_{n}) \mid f \in
      \SSg{F}, \sgn(f) = \blambda_{\rho} \, \}.
    \]
  \end{enumerate}
\end{proposition}

Remark: Note that the union in Proposition~\ref{D5a}\ref{D5a-i} is in
fact a finite union because the increasing sequence
$\rho^{(0)}\subseteq\rho^{(1)}\subseteq\dots\subseteq\rho^{(i)}
\subseteq\dots\subseteq(A^{m})_{s\in S}$ ($m=\ar(\rho)$) must
stabilize after a finite number of steps (since $A$ and $S$ are
finite).
  
\begin{proof}
 \ref{D5a-i}:
 Let
  $\gamma:=\bigcup_{i=0}^{\infty}\rho^{(i)}$. At first we show
  $\gamma\in\SInv F$ (this will imply
  $\Gamma_{F}(\rho)\subseteq\gamma$ by definition of
  $\Gamma_{F}(\rho)$ since $\rho\subseteq\gamma$).

  Indeed, let $f\in F$, $\sgn(f)=(s_{1},\dots,s_{n})$ and
  $r_{j}\in\gamma_{s_{j}s}$ for some $s\in S$, $j\in\{1,\dots,n\}$. We
  have to show $f(r_{1},\dots,r_{n})\in\gamma_{s}$. Since $\gamma$ is
  the union of the increasing sequence $(\rho^{(i)})_{i\in\N}$, there
  must exist $i\in\N$ such that $r_{j}\in\rho^{(i)}_{s_{j}s}$ for all
  $j\in\{1,\dots,n\}$. Consequently, we have
  $f(r_{1},\dots,r_{n})\in\rho^{(i+1)}_{s}\subseteq \gamma_{s}$ by
  definition of $\rho^{(i+1)}_{s}$.

  It remains to show $\gamma\subseteq\Gamma_{F}(\rho)$. We show
  $\rho^{(i)}\subseteq \Gamma_{F}(\rho)$ for all $i\in\N$ by induction
  on $i$. For $i=0$, $\rho^{(0)}=\rho\subseteq\Gamma_{F}(\rho)$ is
  clear. Assume $\rho^{(i)}\subseteq\Gamma_{F}(\rho)$ for some
  $i\in\N$. Then, for $f\in F^{(n)}$ with
  $\sgn(f)=(s_{1},\dots,s_{n})$, $s\in S$,
  $r_{1}\in\rho^{(i)}_{s_{1}s}\subseteq\Gamma_{F}(\rho)_{s_{1}s},\dots,
  r_{n}\in\rho^{(i)}_{s_{n}s}\subseteq\Gamma_{F}(\rho)_{s_{n}s}$, we
  get $f(r_{1},\dots,r_{n})\in\Gamma_{F}(\rho)_{s}$ (because
  $\Gamma_{F}(\rho)$ is $S$\hyp{}invariant), which implies
  $\rho^{(i+1)}_{s}\subseteq\Gamma_{F}(\rho)_{s}$ according to the
  definition of $\rho^{(i+1)}_{s}$. Thus we also have
  $\gamma\subseteq\Gamma_{F}(\rho)$ and therefore, as shown above,
  equality.

  \ref{D5a-ii}: According to \ref{D5a-i} it is enough to show the
  claim for $r\in\rho^{(i)}_{s}$ for each $i\in\N$. This we shall do
  by induction on $i$. For $i=0$ and $r\in\rho^{(0)}_{s}=\rho_{s}$ we
  obviously have $\id(r)=r$, i.e., we can take $q=1$,
  $f=\id\in\SSg{F}$, $r_{1}=r$, $s_{1}=e$.

  Assume that the claim holds for all elements in $\rho^{(i)}_{s}$ and
  all $s\in S$. Let $r\in\rho^{(i+1)}_{s}$. Then (according to the
  definition of $\rho^{(i+1)}$) there exist $f\in F^{(n)}$ with
  $\sgn(f)=(s_{1},\dots,s_{n})$ ($n\in\N_{+}$) and
  $b_{1}\in\rho^{(i)}_{s_{1}s},\dots, b_{n}\in\rho^{(i)}_{s_{n}s}$
  such that $r=f(b_{1},\dots,b_{n})$. By induction hypothesis each
  $b_{j}\in\rho^{(i)}_{s_{j}s}$ can be represented as
  $b_{j}=f_{j}(r_{j1},\dots,r_{jq_{j}})$ with
  $\sgn(f_{j})=(s_{j1},\dots,s_{jq_{j}})$ ($j\in\{1,\dots,n\}$) for
  suitable $r_{jk}\in\rho_{s_{jk}s_{j}s}$
  ($k\in\{1,\dots,q_{j}\}$). Consequently, for
  \begin{multline*}
    h(x_{11},\dots,x_{1q_{1}},\dots,x_{n1},\dots,x_{nq_{n}}) := \\
    f(f_{1}(x_{11},\dots,x_{1q_{1}}),\dots,f_{n}(x_{n1},\dots,x_{nq_{n}}))
  \end{multline*}
  we have
  \begin{align*}
    \sgn(h)&=(s_{11}s_{1},\dots,s_{1q_{1}}s_{1},\dots,s_{n1}s_{n},\dots,s_{nq_{n}}s_{n})
             \quad\text{and} \\
    r&=h(r_{11},\dots,r_{1q_{1}},\dots,r_{n1},\dots,r_{nq_{n}}),
       \quad\text{where $r_{jk} \in \rho_{s_{jk}s_{j}s}$,}
  \end{align*}
  which shows the claim for $\rho^{(i+1)}_{s}$. By induction,
  \ref{D5a-ii} is proved.

  \ref{D5a-iii}: Let $r\in\Gamma_{F}(\rho)_{e}$. According to
  \ref{D5a-ii} (for $s=e$) there exist $f\in\SSg{F}$ with
  $\sgn(f)=(t_{1},\dots,t_{q})$ and $r_{k_{i}}\in\rho_{t_{i}}$,
  $k_{i}\in\{1,\dots,n\}$, $i\in\{1,\dots,q\}$ such that
  $r=f(r_{k_{1}},\dots,r_{k_{q}})$. We can assume that all the
  $r_{k_{i}}$'s on arguments with the same signum are different
  (otherwise the corresponding arguments can be identified according
  to Definition~\ref{C2}\ref{C2-5}). Since $r_{i}\in\rho_{s_{i}}$ (by
  definition), we have $r_{k_{i}}\in \rho_{s_{k_{i}}}$ with
  $s_{k_{i}}=t_{i}$ and therefore
  $\sgn(f)=(t_{1},\dots,t_{q})=(s_{k_{1}},\dots,s_{k_{q}})$.

  For each $r_{j}$ which does not appear among the $r_{k_{i}}$'s, more
  precisely, for each
  $r_{j} \in \rho_{s_{j}} \setminus \{ \, r_{k_{i}} \mid s_{k_{i}} =
  s_{j},\, i \in \{1,\dots,q\} \, \}$
 ($j\in\{1,\dots,n\}$) we add a
  fictitious argument with signum $s_{j}$ to $f$ (according to
  Definition~\ref{C2}\ref{C2-4}). Thus we obtain an $S$\hyp{}operation
  $f'\in \SSg{F}$, the signum of which contains exactly \textsl{all}
  $s_{1},\dots,s_{n}$. With a suitable permutation of the arguments of
  $f'$ (according to Definition~\ref{C2}\ref{C2-2}, \ref{C2-3}, cf.\
  Remark~\ref{C2a}(A)) we finally get an $S$\hyp{}operation with
  $\sgn(f'')=(s_{1},\dots,s_{n})=\blambda_{\rho}$ such that
  $f''(r_{1}, \dots, r_{n})=f(r_{k_{1}},\dots,r_{k_{q}})=r$.
\end{proof}

Proposition~\ref{D6X} below shows that the $S$\hyp{}operations in an
$S$\hyp{}preclone generated by $F$ can be characterized by the
preservation of special $S$\hyp{}relations
$\Gamma_{F}(\chi^{\lambda})$ where the $\chi^{\lambda}$ are defined as
follows.

\begin{definition}\label{D6Xa}
  For a signum $\lambda = (s_{1}, \dots, s_{n}) \in S^{n}$, let
  $\chi^{\lambda} \in \SRela[k^{n}](A)$ be defined by
  \[
    \chi^{\lambda}_{s} :=
 \{ \, \kappa_{i} \mid s_{i} = s,\, i \in
    \{1, \dots, n \} \, \}
 \quad \text{for } s \in S.
  \]

  In the above, for $k := \lvert A \rvert$ and fixed $n\in\N_{+}$, the
  tuples $\kappa_{i}$ are defined as in the ``classical'' case
  (\cite[1.1.16]{PoeK79}): $(\kappa_{1}, \dots, \kappa_{n})$ is the
  $(k^{n}\times n)$-matrix with columns
  $\kappa_{1}, \dots, \kappa_{n}$ such that the rows are
  \underline{\textsl{all}} $n$-tuples from $A^{n}$ (may be ordered
  lexicographically). Thus we have
  $\kappa_{i}\in\chi^{\lambda}_{s_{i}}$. With the notation introduced
  in Definition~\ref{C3} we have $\lambda=\blambda_{\chi^{\lambda}}$.

\end{definition}

\begin{proposition}\label{D6X}
  Let $F \subseteq \SOp(A)$, let
  $\lambda = (s_{1}, \dots, s_{n}) \in S^{n}$ be a signum and let
  $g \in \SOpa[n](A)$ with $\sgn(g) = \lambda$. Then we have
  \[
 g \in \SSg{F} \iff g \Spreserves \Gamma_{F}(\chi^{\lambda}).
  \]
\end{proposition}

\begin{proof}
  The implication ``$\Longrightarrow$'' is clear since
  $\Gamma_{F}(\chi^{\lambda})$ is invariant (by definition) for every
  $S$\hyp{}operation in $\SSg{F}$ (cf.\ Lemma~\ref{D3A}).

  ``$\Longleftarrow$'': $g\Spreserves \Gamma_{F}(\chi^{\lambda})$
  implies that
  $g(\Gamma_{F}(\chi^{\lambda})_{s_{1}},\dots,\Gamma_{F}(\chi^{\lambda})_{s_{n}})\subseteq\Gamma_{F}(\chi^{\lambda})_{e}$
  (cf.~\ref{D1}\eqref{D1i}), in particular
  $g(\kappa_{1},\dots,\kappa_{n})\in\Gamma_{F}(\chi^{\lambda})_{e}$. According
  to Proposition~\ref{D5a}\ref{D5a-iii} (for $u=e$) we have
  \[
    \Gamma_{F}(\chi^{\lambda})_{e} =
 \{ \,
    f(\kappa_{1},\dots,\kappa_{n}) \mid f \in \SSg{F}, \, \sgn(f) =
    \lambda \, \}.
  \]
  Hence there exists some $f\in\SSg{F}^{(\lambda)}$ such that
  $g(\kappa_{1},\dots,\kappa_{n})=f(\kappa_{1},\dots,\kappa_{n})$. Thus
  $g$ and $f$ agree on each element of $A^{n}$ (the rows of
  $\chi^{\lambda}$), i.e., $g=f\in\SSg{F}$.
\end{proof}


\section{The Galois closures for the Galois connection
  $\boldsymbol{\SPol}$--$\boldsymbol{\SInv}$}\label{secE}

Now we are able to characterize the Galois closures of the Galois
connection $\SPol$--$\SInv$. Recall that throughout the paper we
assume that $A$ and $S$ \textbf{are finite}.

\begin{theorem}\label{Thm-I} Let $S$ be an arbitrary monoid. Then,
  for $F\subseteq\SOp(A)$, we have
  \[
    \SSg{F} =\SPol\SInv F,
  \]
  i.e., the Galois closure is the $S$\hyp{}preclone generated by $F$.
\end{theorem}

\begin{proof} Since $\SPol\SInv$ is a closure operator,
 we have
  the inclusion
 $\SSg{F} \subseteq \linebreak \SPol\SInv \SSg{F}=\SPol\SInv F$
  (the last equality follows from Lemma~\ref{D3A}). For the converse
  inclusion let $g\in\SPol\SInv F$ with $\sgn(g)=\lambda$. Then
  $g\in\SPol \Gamma_{F}(\chi^{\lambda})$ (since
  $\Gamma_{F}(\chi^{\lambda})$ is invariant by Definition \ref{D5})
  and we get $g\in \SSg{F}$ by Proposition~\ref{D6X}.
\end{proof}

In Theorem~\ref{Thm-II} we shall characterize the Galois closed
$S$\hyp{}relational clones. In preparation of the proof we need
several lemmata.

\begin{lemma}[{cf.\ \cite[Lemma 1.2.2]{PoeK79}}]\label{D7}
  Let $F \subseteq \SOp(A)$. Then each invariant
  $\rho \in \SInva[m] F$ \textup{(}$m\in\N_{+}$\textup{)} can be
  obtained from $\Gamma_{F}(\chi^{\blambda_{\rho}})$ by
  projections. Consequently,
  $\SInv F = \SRg{\{ \, \Gamma_F(\chi^{\blambda_\rho}) \mid \rho \in
    \SRel(A) \, \}}$.
\end{lemma}

\begin{proof}According to Definition~\ref{C3}, for
  $\blambda_{\rho}=(s_{1},\dots,s_{n})$, we have
  $\rho_{s} = \{ \, r_{i} \mid s_{i} = s \, \}$.
 Thinking of $\rho$
  as the $m \times n$ matrix $M=(r_{1}, \dots, r_{n})$
  ($r_{i}\in A^{m}$), we see that each row of $M$ appears as a row of
  the $|A|^n\times n$ matrix
  $\chi^{\blambda_\rho}=(\kappa_1, \dots, \kappa_n)$ say, the $j$-th
  row of $\rho$ is the $z_j$-th row of $\chi^{\blambda_\rho}$
  ($j \in \{1, \dots, m\}$). Then
  $\rho = \pr_{z_1, \dots, z_m}(\chi^{\blambda_\rho})$. By
  Lemma~\ref{LemGammaF}, we have
  $\Gamma_F(\rho) = \pr_{z_1, \dots,
    z_m}(\Gamma_{F}(\chi^{\blambda_\rho}))$. Since $\rho \in \SInv F$,
  we get
  $\rho = \Gamma_F(\rho) = \pr_{z_1, \dots,
    z_m}(\Gamma_F(\chi^{\blambda_\rho}))$. Therefore,
  $\SInv F \subseteq \SRg{\{ \, \Gamma_F(\chi^{\blambda_\rho}) \mid
    \rho \in \SRel(A) \, \}}$ by Remark~\ref{C5}(B). The inclusion
  $\SRg{\{ \, \Gamma_F(\chi^{\blambda_\rho}) \mid \rho \in \SRel(A) \,
    \}} \subseteq \SInv F$ holds, because
  $\{ \, \Gamma_F(\chi^{\blambda_\rho}) \mid \rho \in \SRel(A) \, \}
  \subseteq \SInv F$ and $\SInv F$ is an $S$\hyp{}relational clone by
  Lemma~\ref{D3A}.
\end{proof}

\begin{notation}\label{N}
  For given $Q\subseteq \SRel(A)$ and $\rho\in\SRel(A)$ let
  \[
    \gamma(\rho) := \gamma_{Q}(\rho) := \bigcap \{ \, \rho' \in
    \SRg{Q} \mid \rho \subseteq \rho' \,\}
  \]
  be the smallest $S$\hyp{}relation in $\SRg{Q}$ that has the same
  arity as $\rho$ and contains $\rho$ (the index $Q$ for $\gamma_{Q}$
  is omitted if the $Q\subseteq\SRel(A)$ under consideration is
  fixed). Note that $\gamma(\rho)\in\SRg{Q}$ because $\SRg{Q}$ is an
  $S$\hyp{}relational clone and therefore closed under intersections
  (cf.\ Definition~\ref{C4}\ref{C4-6}).
\end{notation}

  \begin{lemma}\label{L0}
    Let $Q\subseteq\SRel(A)$ and $F:=\SPol Q$. If
    $\Gamma_{F}(\chi^{\lambda})=\gamma(\chi^{\lambda})$ for each
    signum $\lambda=(s_{1},\dots,s_{n})$ ($s_{1},\dots,s_{n}\in S$,
    $n\in\N_{+}$), then $\SRg{Q} = \SInv \SPol Q$.
  \end{lemma}

  \begin{proof}
    By Lemma~\ref{D3A} and the general properties of Galois
    connections, we have
    $\SRg{Q} \subseteq \SInv \SPol \SRg{Q} = \SInv F$. In order to
    prove the converse inclusion, it suffices to show that
    $\Gamma_F(\chi^\lambda) \in \SRg{Q}$ for all signa $\lambda$,
    because this implies, by Lemma~\ref{D7}, that
    \begin{align*}
      \SInv F
      & = \SRg{\{ \, \Gamma_F(\chi^{\blambda_\rho}) \mid \rho \in \SRel(A) \, \}}
      \\ & \subseteq \SRg{\{ \, \Gamma_F(\chi^\lambda) \mid \lambda \in S^* \, \}}
           \subseteq \SRg{\SRg{Q}}
           = \SRg{Q}.
    \end{align*}
    Furthermore, $\gamma(\chi^{\lambda})\in \SRg{Q}$ by Definition
    \ref{N}. Thus the assumption
    $\Gamma_{F}(\chi^{\lambda})=\gamma(\chi^{\lambda})$ implies
    $\Gamma_{F}(\chi^{\lambda})\in \SRg{Q}$ and we are done.
  \end{proof}

  For the next lemmata we always assume that $Q\subseteq\SRel(A)$ is
  arbitrarily chosen but fixed and that $F:=\SPol Q$.

\begin{lemma}\label{L1} We have
  $\Gamma_{F}(\chi^{\lambda})_{e}=\gamma(\chi^{\lambda})_{e}$ for all
  signa $\lambda=(s_{1},\dots,s_{n})$
  \textup{(}$n\in\N_{+}$, $s_{1},\dots, s_{n}\in S$\textup{)}.
\end{lemma}

\begin{proof} Recall that
  $\chi^\lambda_s = \{ \, \kappa_i \mid s_i = s,\, i \in \{1, \dots,
  n\} \, \}$
 for $s\in S$ and $\lambda=(s_{1},\dots,s_{n})$, where
  $(\kappa_1, \dots, \kappa_n)$ is an $\lvert A \rvert^n\times n$
  matrix with columns $\kappa_1, \dots, \kappa_n$ such that the rows
  are all $n$-tuples from $A^n$.

  Let $\gamma:=\gamma(\chi^{\lambda})$ and assume that there exists
  $\bfr \in \gamma_e \setminus \Gamma_F(\chi^\lambda)_e$ (and we are
  going to show that this leads to a contradiction). By
  Proposition~\ref{D5a}\ref{D5a-iii} (take
  $(r_{1},\dots,r_{n})=(\kappa_{1},\dots,\kappa_{n})$ and note
  $e\in I(S)$), the function $f_{\bfr}$ of signum $\lambda$ that is
  defined by $f_{\bfr}(\kappa_1, \dots, \kappa_n) = \bfr$ does not
  belong to $F = \SPol Q$. Therefore there exists an $S$\hyp{}relation
  $\theta \in Q$, say $m$-ary, that is not $S$\hyp{}preserved by
  $f_{\bfr}$, i.e., there exist a $v \in S$ and tuples
  $\bfr_i \in \theta_{s_i v}$ ($i \in \{1, \dots, n\}$) such that
  $\mathbf{a} := f_{\bfr}(\bfr_1, \dots, \bfr_n) \notin \theta_v$. By
  index translation by $v$, we obtain the $S$\hyp{}relation
  $\theta^* := \mu_{v}(\theta) \in \SRg{Q}$ with
  $\theta^*_s = \theta_{sv}$ ($s \in S$) that is not
  $S$\hyp{}preserved by $f_{\bfr}$ either, because
  $\bfr_i \in \theta_{s_i v} = \theta^*_{s_i}$
  ($i \in \{1, \dots, n\}$) and
  $\mathbf{a} \notin \theta_v = \theta^*_e$.

  Consider the matrix $M = (\bfr_1, \dots, \bfr_n)$. The rows of $M$
  occur as rows of \linebreak $(\kappa_1, \dots, \kappa_n)$; say the $j$-th row
  of $M$ equals the $h_j$-th row of $\chi^\lambda$. Let
  $(\delta_{\tau}^{q+m})_{s\in S}$ be the diagonal relation with
  $\tau = \{ \, (h_j, q + j) \mid j \in \{1, \dots, m\} \, \}$, and
  let
  $\theta' := (\gamma \times \theta^{*}) \wedge \delta^\tau_{q+m}$. In
  other words, $\theta'_s$ comprises those tuples from
  $(\gamma \times \theta^{*})_s$ whose $h_j$-th and $(q + j)$-th
  components are equal, for $j \in \{1, \dots, m\}$. By removing the
  last $m$ rows, we obtain the $S$\hyp{}relation
  $\theta'' := \pr_{1, \dots, q}(\theta')$. Since the tuple
  $\kappa_i \times \bfr_i$ belongs to
  $(\gamma \times \theta^{*})_{s_i}$ and hence also to
  $\theta'_{s_i}$, for all $i \in \{1, \dots, n\}$, we have
  $\kappa_i \in \theta''_{s_i}$ for all $i \in \{1, \dots, n\}$, i.e.,
  $\chi^\lambda \subseteq \theta''$. Moreover,
  $\theta'' \subseteq \gamma$ holds by construction. The
  $S$\hyp{}relation $\theta''$ was built from relations in $Q$ by
  using operations described in Definition~\ref{C4} and
  Remark~\ref{C5}; therefore $\theta'' \in \SRg{Q}$.

  We show that $\bfr\notin \theta''_{e}$. Indeed, if
  $\bfr\in \theta''_{e}$ then there would exist a
  $\mathbf{b}\in \theta^{*}$ such that
  $\bfr\times \mathbf{b}\in \delta^\tau_{q+m}$ (recall the definitions
  of $\theta'$ and $\theta''$, and that
  $\bfr = f_{\bfr}(\kappa_1, \dots, \kappa_n) \in \gamma_e$ thus
  $\bfr\times \mathbf{b}\in \gamma_{e}\times \theta^{*}$),
  consequently
  \begin{align*}
 \mathbf{b}
 &= \pr_{h_{1},\dots,h_{m}}(\bfr) =
    \pr_{h_{1},\dots,h_{m}}(f_{\bfr}(\kappa_{1},\dots,\kappa_{n})) \\
    & = f_{\bfr}(\pr_{h_{1},\dots,h_{m}}(\kappa_{1}),\dots,
    \pr_{h_{1},\dots,h_{m}}(\kappa_{n}))
 =
    f_{\bfr}(\bfr_{1},\dots,\bfr_{n})=\mathbf{a}\notin \theta^{*},
  \end{align*}
  a contradiction.

  Thus $\theta''_e \subsetneqq \gamma_e$ and so
  $\theta'' \subsetneqq \gamma$. We conclude that $\gamma$ is not the
  smallest $S$\hyp{}relation in $\SRg{Q}$ containing
  $\chi^\lambda$. We have reached the desired contradiction.
\end{proof}

The following theorem generalizes the characterization of (usual)
relational clones (cf.\ \cite[Satz 1.2.3]{PoeK79}) to
$S$\hyp{}relational clones.

\begin{theorem}\label{Thm-II}
  Let $S$ be an arbitrary monoid. Then, for $Q\subseteq\SRel(A)$, we
  have
  \begin{align*}
    \SRg{Q} = \SInv \SPol Q,
  \end{align*}
  i.e., the Galois closure is the $S$\hyp{}relational clone generated
  by $Q$.
\end{theorem}

\begin{proof}
  Again we use the notation $F:=\SPol Q$ and $\gamma$ as in the
  lemmata before. Because of Lemma~\ref{L0} it is enough to show
  $\Gamma_{F}(\chi^{\lambda})_{v}=\gamma(\chi^{\lambda})_{v}$ for each
  $v\in S$ and arbitrary signa $\lambda$.

  By Lemma~\ref{L1} this is true for $v=e$. Thus let
  $v\in S\setminus\{e\}$ and let $\lambda=(t'_{1},\dots,t'_{p})$. We
  have to show
  $\Gamma_{F}(\chi^{\lambda})_{v}=\gamma(\chi^{\lambda})_{v}$. Note
  that the entries of $\lambda$ could be permuted arbitrarily without
  changing this equality (this follows from the fact that
  $S$\hyp{}preclones are closed under permutation of arguments, cf.\
  Remark~\ref{C2a}).

  Let $\Lambda=\{t_{1},\dots,t_{n}\}$ be the set of all different
  entries in the signum $\lambda$, i.e.,
  $\Lambda=\{t'_{1},\dots,t'_{p}\}$. Ordering the entries of $\lambda$
  correspondingly we can assume that
  \begin{align*}
    \lambda=(\underbrace{t_{1},\dots,t_{1}}_{\ell_{1}},\dots,
    \underbrace{t_{i},\dots,t_{i}}_{\ell_{i}},\dots,
    \underbrace{t_{n},\dots,t_{n}}_{\ell_{n}}),
  \end{align*}
  where $t_{i}$ appears $\ell_{i}$ times ($i\in\{1,\dots,n\}$), thus
  $p=t_{1}+ \ldots + t_{n}$. The corresponding elements (columns) of
  $\chi^{\lambda}$ are denoted as follows
  \begin{align*}
    \chi^{\lambda}=(\underbrace{\kappa_{t_{1},1},\dots,
    \kappa_{t_{1},\ell_{1}}}_{\chi^{\lambda}_{t_{1}}},\dots,
    \underbrace{\kappa_{t_{i},1},\dots,
    \kappa_{t_{i},\ell_{i}}}_{\chi^{\lambda}_{t_{i}}},\dots,
    \underbrace{\kappa_{t_{n},1},\dots,
    \kappa_{t_{n},\ell_{n}}}_{\chi^{\lambda}_{t_{n}}})
  \end{align*}

  Let $\cM^{v}=(M_{s})_{s\in S}$ with
  $M_{s} = \{ \, x \in S \mid xv = s \, \}$ be the family as defined
  in \ref{LM} (for simplicity we write $M_{s}$ instead of
  $M^{v}_{s}$). 
 Now we consider the signum
  \begin{align*}
    \lambdahat={}& (
                   \overbrace{\hspace{1.3pt}\dots,
                   \underbrace{s_{1},\dots,s_{1}}_{\ell_{1}},\dots\hspace{1.3pt}}^{s_{1}\in{
                   M_{t_{1}}}},\dots,
                   \overbrace{\hspace{1.3pt}\dots,
                   \underbrace{s_{i},\dots,s_{i}}_{\ell_{i}},\dots\hspace{1.3pt}}^{s_{i}\in{
                   M_{t_{i}}}},\dots,
                   \overbrace{\hspace{1.3pt}\dots,
                   \underbrace{s_{n},\dots,s_{n}}_{\ell_{n}},\dots\hspace{1.3pt}}^{s_{n}\in{
                   M_{t_{n}}}})\\
    =:{}& (v_{1},\dots,\dots,\dots,\dots,\dots,\dots,\dots,\dots,\dots,\dots,\dots,\dots,\dots,\dots, v_{\widehat{n}}),
  \end{align*}
  where each $s\in M_{t_{i}}$ appears exactly $\ell_{i}$ times
  ($i\in\{1,\dots,n\}$). Because of \ref{LM}\ref{LMb} no $s$ can
  appear in different $M_{t_{i}}$.

  At first consider the case $\lambdahat=\emptyset$, i.e.,
  $M_{t_{1}}=\ldots=M_{t_{n}}=\emptyset$, consequently (by the
  definition of $M_{t}$) we have
  $\{t_{1},\dots,t_{n}\}\cap Sv=\emptyset$. Let $m:=|A^{p}|$ be the
  arity of $\chi^{\lambda}$, and let $d\in \SD_{A}$ denote the $m$-ary
  diagonal relation (cf.\ Definition~\ref{C4a}) with $d_{s}=\emptyset$
  for $s\in Sv$ and $d_{s}= A^{m}$ otherwise. We get
  \begin{align*}
    (\gamma(\chi^{\lambda})\land d)_{s}=
    \begin{cases}
      \gamma(\chi^{\lambda})_{s}\land \emptyset= \emptyset&\text{if }s\in Sv,\\
      \gamma(\chi^{\lambda})_{s}\land
      A^{m}=\gamma(\chi^{\lambda})_{s}&\text{otherwise.}
    \end{cases}
  \end{align*}
  Then
  $\chi^{\lambda}\subseteq \rho:=\gamma(\chi^{\lambda})\land
  d\in\SRg{Q}$, consequently $\gamma(\chi^{\lambda})\subseteq \rho$,
  in particular we have
  $\gamma(\chi^{\lambda})_{v}\subseteq\rho_{v}=\emptyset$. Therefore
  $\Gamma_{F}(\chi^{\lambda})_{v}=\gamma(\chi^{\lambda})_{v}=\emptyset$,
  and we are done.

  Thus, from now on, we can assume that $M_{t}\neq \emptyset$ for at
  least one $t\in\{t_{1},\dots,t_{n}\}$. The $S$\hyp{}relation
  $\chi^{\lambdahat}$ consists of
  $\widehat{n}:=\ell_{1}\cdot|M_{t_{1}}|+\ldots+\ell_{n}\cdot|M_{t_{n}}|\geq
  1$ columns which are denoted by
  \begin{align*}
    \chi^{\lambdahat}=(\overbrace{\hspace{1.3pt}\dots,\kappahat_{s_{1},1},\dots,\kappahat_{s_{1},\ell_{1}},\dots\hspace{1.3pt}}^{s_{1}\in{
    M_{t_{1}}}}\,,\;\dots\;,\,\overbrace{\hspace{1.3pt}
    \dots,\kappahat_{s_{n},1},\dots,\kappahat_{s_{n},\ell_{n}},\dots}^{s_{n}\in{
    M_{t_{n}}}}),
  \end{align*}
  i.e.,
  $\chi^{\lambdahat}_{s}=\{\kappahat_{s,1},\dots,\kappahat_{s,\ell_{i}}\}$
  for $s\in M_{t_{i}}$, $i\in\{1,\dots,n\}$.

  Let $\widehat{m}:=|A|^{\widehat{n}}$. Note that the colums of
  $\chi^{\lambda}$ and $\chi^{\lambdahat}$ are elements of $A^{m}$ and
  $A^{\widehat{m}}$, respectively.

  \begin{figure}
    \begin{center}
      \scalebox{0.845}{\includegraphics{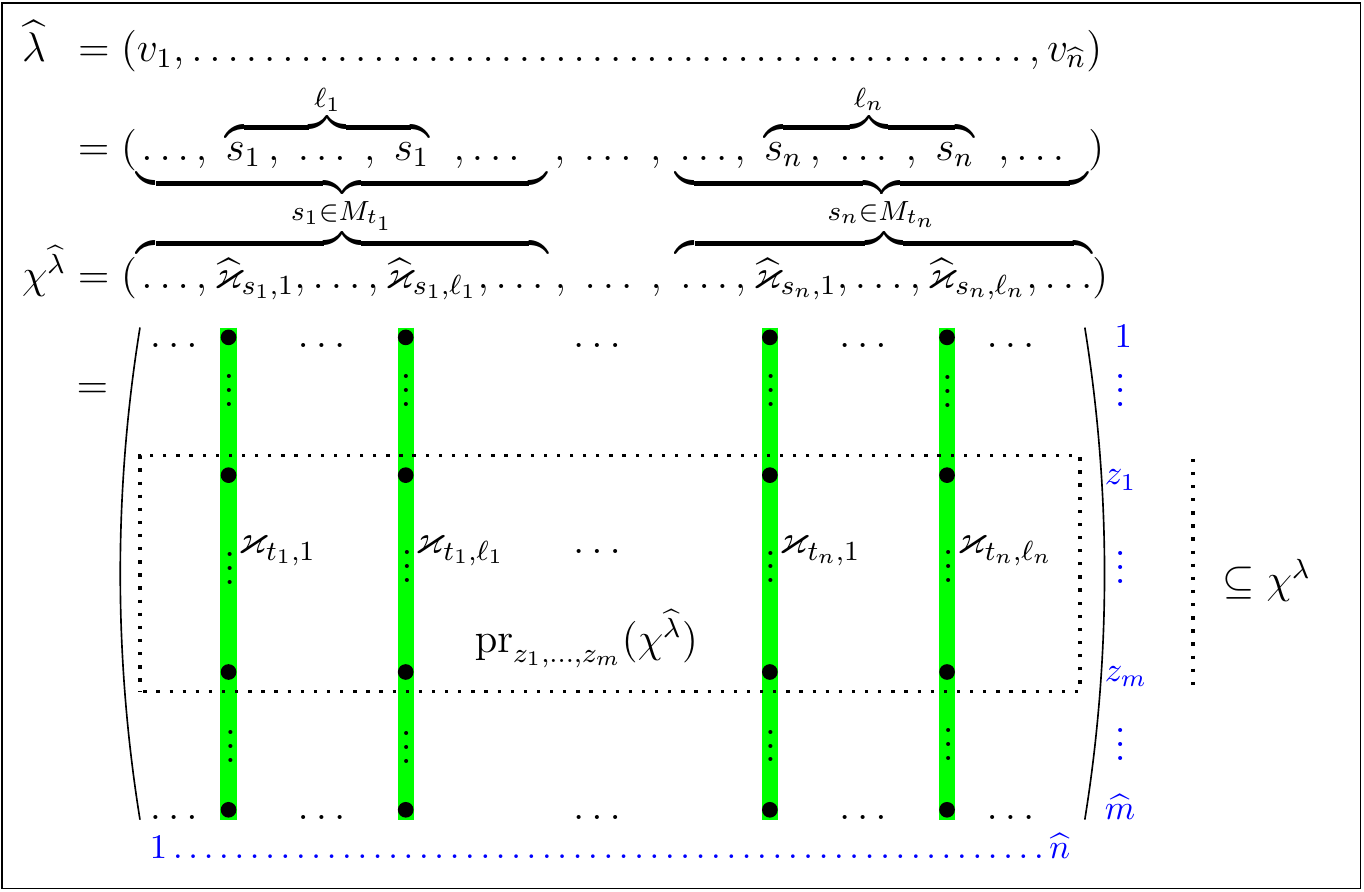}}
    \end{center}
    \caption{The signum $\lambdahat$ and the corresponding columns of
      the $S$\hyp{}relation $\chi^{\lambdahat}$ with the
 projection
      $\pr_{z_{1},\dots,z_{m}}(\chi^{\lambdahat})\subseteq\chi^{\lambda}$}
    \label{fig1}
  \end{figure}

  Choose the rows of $\chi^{\lambdahat}$ with indices
  $z_{1},\dots,z_{m}$ such that (for notation, see Remark~\ref{C5}(B))
  \begin{align*}
    \pr_{z_{1},\dots,z_{m}}(\chi^{\lambdahat})=(\overbrace{\hspace{1.3pt}\dots,
    \kappa_{t_{1},1},\dots,\kappa_{t_{1},\ell_{1}},\dots\hspace{1.3pt}}^{s_{1}\in{
    M_{t_{1}}}}\,,\;\dots\;,\,\overbrace{\hspace{1.3pt}
    \dots,\kappa_{t_{n},1},\dots,\kappa_{t_{n},\ell_{n}},\dots}^{s_{n}\in{
    M_{t_{n}}}}),
  \end{align*}
  i.e., $\pr_{z_{1},\dots,z_{m}}(\kappahat_{s,j})=\kappa_{t_{i},j}$
  for $s\in M_{t_{i}}$, $j\in\{1,\dots,\ell_{i}\}$, therefore
  \begin{align*}\tag{°}
    (\pr_{z_{1},\dots,z_{m}}(\chi^{\lambdahat}))_{s}
    = \pr_{z_{1},\dots,z_{m}}(\chi^{\lambdahat}_{s})
    = \{\kappa_{t_{i},1},\dots,\kappa_{t_{i},\ell_{i}}\}
    = \chi^{\lambda}_{t_{i}}
  \end{align*}
  for $s\in M_{t_{i}}$, $i\in\{1,\dots,n\}$. Note that the rows
  $z_{1},\dots,z_{m}$ must exist because the rows of
  $\chi^{\lambdahat }$ are exactly \textsl{all} tuples in
  $A^{\widehat{m}}$ (cf.\ Definition~\ref{D6Xa}). The structure of
  $\lambdahat$ and $\chi^{\lambdahat}$ and the introduced notation can
  be seen schematically in Figure~\ref{fig1}.

  We define $\rho:=\pr_{z_{1},\dots,z_{m}}(\gamma(\chi^{\lambdahat}))$
  and $\rhobar:=\sqcap_{\cM^{v}}\rho=\sqcap^{v}\rho$ (the last
  equation follows from~\ref{LM}). Since
  $\gamma(\chi^{\lambdahat})\in\SRg{Q}$, also $\rhobar$ belongs to
  $\SRg{Q}$, moreover we have $\rhobar_{v}\subseteq \rho_{e}$ because
  of property \ref{LM}\ref{LMa} (cf.\ Definition~\ref{C4}\ref{C4-8} or
  \ref{C4-Self}\ref{C4-9}).

  In particular, $\chi^{\lambda}\subseteq \rhobar$ because (with the
  above\hyp{}mentioned equality (°)) we have
  $\chi^{\lambda}_{t}=(\pr_{{z_{1},\dots,z_{m}}}(\chi^{\lambdahat}))_{s}
  \subseteq (\pr_{{z_{1},\dots,z_{m}}}(\gamma(\chi^{\lambdahat}))_{s}
  $ for all $s\in M_{t}$, $t\in \Lambda$ (note that
  $\rhobar_{t} = \bigcap \{ \, \rho_{s} \mid s \in M_{t} \,
  \}$). Consequently, $\gamma(\chi^{\lambda})\subseteq \rhobar$ since
  $\gamma(\chi^{\lambda})$ is the least $S$\hyp{}relation in $\SRg{Q}$
  which contains $\chi^{\lambda}$.

  Now, for $\Gamma:=\Gamma_{F}(\chi^{\lambda})$ we can prove
  $\Gamma_{v}=\gamma(\chi^{\lambda})_{v}$. We have to show only
  $\gamma(\chi^{\lambda})_{v}\subseteq \Gamma_{v}$ (the other
  inclusion is always fulfilled).

  Let $\bfr\in \gamma(\chi^{\lambda})_{v}$. Then $\bfr\in \rho_{e}$
  since
  $\gamma(\chi^{\lambda})_{v}\subseteq \rhobar_{v}\subseteq
  \rho_{e}$. By the definition of $\rho$ there must exist
  $\mathbf{\widehat{r}}\in \gamma(\chi^{\lambdahat})_{e}$ such that
  $\pr_{{z_{1},\dots,z_{m}}}(\mathbf{\widehat{r}})=\bfr$. By
  Lemma~\ref{L1}, $\Gamma_{e}=\gamma(\chi^{\lambda})_{e}$, thus
  $\mathbf{\widehat{r}}\in \Gamma_{e}$. According to
  Proposition~\ref{D5a}\ref{D5a-iii} (with $\rho=\chi^{\lambdahat}$,
  $r=\mathbf{\widehat{r}}$) there exist $f\in F$ with
  $\sgn(f)=\lambdahat=(v_{1},\dots,v_{\widehat{n}})$ and
  $\widehat{\bfr}_{j}\in \chi^{\lambdahat}_{v_{j}}$
  ($j\in\{1,\dots,\widehat{n}\}$) such that
  $\widehat{\bfr}
  =f(\widehat{\bfr}_{1},\dots,\widehat{\bfr}_{\widehat{n}})$.

  By the definition of $\lambdahat$, each $v_{j}$ belongs to $M_{t}$
  for some $t\in \Lambda$, consequently
  $\pr_{{z_{1},\dots,z_{m}}}(\chi^{\lambdahat}_{v_{j}})=\chi^{\lambda}_{t}=\chi^{\lambda}_{v_{j}v}$,
  the last equality follows from condition \ref{LM}\ref{LMc}.

  Applying $\pr_{{z_{1},\dots,z_{m}}}$ to
  $\widehat{\bfr}
  =f(\widehat{\bfr}_{1},\dots,\widehat{\bfr}_{\widehat{n}})$, we get
  $\bfr=f(\bfr_{1},\dots,\bfr_{\widehat{n}})$ where
  $\bfr_{j}=\pr_{{z_{1},\dots,z_{m}}}(\widehat{\bfr}_{j})$
  ($j\in\{1,\dots,\widehat{n}\}$) is an element of
  $\pr_{{z_{1},\dots,z_{m}}}(\chi^{\lambdahat}_{v_{j}})
  =\chi^{\lambda}_{v_{j}v}$. Thus
  $\bfr\in
  f(\chi^{\lambda}_{v_{1}v},\dots,\chi^{\lambda}_{v_{\widehat{n}}v})\subseteq
  f(\Gamma_{v_{1}v},\dots,\Gamma_{v_{\widehat{n}}v})\subseteq
  \Gamma_{v}=\Gamma_{F}(\chi^{\lambda})_{v}$ (the last inclusion
  follows from $f\Spreserves\Gamma_{F}(\chi^{\lambda})$, cf.\
  \ref{D1}\eqref{D1i}, because $f\in F$). Since $\bfr$ was chosen
  arbitrarily,
  $\gamma(\chi^{\lambda})_{v}\subseteq\Gamma_{F}(\chi^{\lambda})_{v}$
  and we are done.
\end{proof}

As a final remark to conclude this section, we would like to present a
connection between $S$\hyp{}preclones and minions (and between the
corresponding Galois connections).

For any $\lambda \in S^n$, let
$F^{(\lambda)} := \{ \, f \in F \mid \sgn(f) = \lambda \, \}$ (the
part of $F$ with signum $\lambda$). For $s \in S$, let
$F^{[s]} := \bigcup_{\lambda \in W_s} F^{(\lambda)}$, where
$W_s := \{ \, (\underbrace{s, \dots, s}_n) \mid n \in \N_{+} \, \}$
(the $S$\hyp{}signed part of $F$).

Let $f \colon A^n \to B$ and $g \colon A^m \to B$. We say that $f$ is
a \emph{minor} of $g$, and we write $f \leq g$, if there exists a map
$\sigma \colon \{1, \dots, m\} \to \{1, \dots, n\}$ such that
$f(a_1, \dots, a_n) = g(a_{\sigma(1)}, \dots, a_{\sigma(m)})$ for all
$a_1, \dots, a_n \in A$. The minor relation $\leq$ is a quasiorder (a
reflexive and transitive relation) on the set
$\Fun(A,B) := \bigcup_{m\in\N_{+}}B^{A^{m}}$. Downwards closed subsets
of $(\Fun(A,B), \mathord{\leq})$ are called \emph{minor\hyp{}closed
  classes} or \emph{minions}.

Minor\hyp{}closed classes can be characterized with a Galois
connection analogous to $\Pol$--$\Inv$ (see
Definition~\ref{def:Pol-Inv}). Instead of relations, the dual objects
are now \emph{relation pairs}, i.e., pairs $(\rho,\rho')$, where
$\rho \in \Rela[m](A)$ and $\rho' \in \Rela[m](B)$ for some
$m \in \N$. Denote by $\RelP(A,B)$ the set of all such relation
pairs. Let $f \in \Fun(A,B)$ and $(\rho,\rho') \in \RelP(A,B)$. We say
that $f$ \emph{preserves} $(\rho,\rho')$, and we write
$f \preserves (\rho,\rho')$, if
$f(\rho, \dots, \rho) \subseteq \rho'$. For $F \subseteq \Fun(A,B)$
and $Q \subseteq \RelP(A,B)$, we define
\begin{align*}
  \PolP Q &:= \{ \, f \in \Fun(A,B) \mid \forall\, (\rho,\rho') \in Q \colon f \preserves (\rho,\rho') \, \}
  && \text{(polymorphisms)}, \\
  \InvP F &:= \{ \, (\rho,\rho') \in \Rel(A,B) \mid \forall\, f \in F \colon f \preserves (\rho,\rho') \, \}
  && \text{(invar.\ rel.\ pairs)}.
\end{align*}
It was shown by Pippenger~\cite[Theorem~2.1]{Pip2002} that if $A$ and
$B$ are finite, then a set $F \subseteq \Fun(A,B)$ is a minion if and
only if $F = \PolP Q$ for some $Q \subseteq \RelP(A,B)$. This was
later generalized by Couceiro and
Foldes~\cite[Theorem~2.1]{CouFol2005} for functions defined and valued
on arbitrary sets; in this case, the Galois closed sets of functions
are the locally closed minions.

For $s \in S$, let
$F^{[s]} := \{ \, f \in \SOp(A) \mid \Sgn(f) = \{s\} \, \}$ be the
\New{$S$\hyp{}signed part of $F$}, i.e., all $S$\hyp{}operations in
$F$ where each argument has signum $s$.

\begin{proposition}\label{D8}
  If $F$ is an $S$\hyp{}preclone on $A$, then for each $s \in S$,
  $F^{[s]}$ is a minion; in particular, for the neutral element $e$ of
  $S$, $F^{[e]}$ is a clone. Moreover if $F = \SPol Q$ for some set
  $Q \subseteq \SRel(A)$, then
  $F^{[s]} = \Pol \{ \, (\rho_{st}, \rho_t) \mid \rho \in Q, \, t \in
  S \, \}$;
 in particular,
  $F^{[e]} = \Pol \{ \, \rho_t \mid \rho \in Q, t \in S \, \}$.
\end{proposition}

\begin{proof}
  It is clear from the definition that $F^{[s]}$ is minor-closed; in
  particular, it is closed under arbitrary identification of
  arguments. For the neutral element $e$, we furthermore have that
  $F^{[e]}$ is closed under composition; hence $F^{[e]}$ is a clone
  (this also follows from Lemma~\ref{D2}\ref{D2b}). Moreover, an
  $S$\hyp{}operation $f$ with $\sgn(f) = (s, \dots, s)$
  $S$\hyp{}preserves $\rho = (\rho_s)_{s \in S}$ if and only if
  $f(\rho_{st}, \dots, \rho_{st}) \subseteq \rho_t$ for all $t \in
  S$. The latter is equivalent to the condition that the operation $f$
  (ignoring signum) preserves the relation pair $(\rho_{st}, \rho_t)$
  for all $t \in S$. Therefore, if $F = \SPol Q$, then (ignoring
  signa)
  \begin{align*}
    F^{[s]}
    & = \{ \, f \in \SOp(A) \mid \sgn(f) = (s, \dots, s), \, f \Spreserves Q \, \}
    \\ & = \{ \, f \in \Op(A) \mid f \preserves \{ \, (\rho_{st}, \rho_t)
         \mid \rho \in Q, \, t \in S \, \}  \, \} 
    \\ & = \PolP \{ \, (\rho_{st}, \rho_t) \mid \rho \in Q, \, t \in S \, \}.
  \end{align*}
  For the neutral element $e$ of $S$, we get furthermore (see also
  Lemma~\ref{D2}\ref{D2b}) that
  \[
    F^{[e]} = \PolP \{ \, (\rho_t, \rho_t) \mid \rho \in Q, \, t \in S
    \, \} = \Pol \{ \, \rho_t \mid \rho \in Q, t \in S \,
    \}. \tag*{\qedhere}
  \]
\end{proof}


\section{The lattice $\SL_{A}$ of
  $\boldsymbol{S}$\hyp{}preclones} \label{secH}

The $S$\hyp{}preclones on a set $A$ form a lattice $\SL_{A}$ (with
respect to inclusion). The least element is $\SJ_{A}$ (cf.\
Remark~\ref{C2a}) and the largest element is $\SOp(A)$. In this
section we deal with atoms and coatoms of $\SL_{A}$, with embeddings
of the lattice $\cL_{A}$ of (usual) clones on $A$ into $\SL_{A}$, and
with some inner symmetries of $\SL_{A}$. But at first we shall look
for generating systems of the $S$\hyp{}preclone $\SOp(A)$ and the
$S$\hyp{}relational clone $\SRel(A)$.

\begin{proposition}\label{H0}
  \leavevmode
  \begin{enumerate}[label=\textup{(\arabic*)}]
  \item \label{H0-1} The $S$\hyp{}preclone $\SOp(A)$ is finitely
    generated.

    For instance, for $A=\{0,1,\dots,k-1\}$ we have
    \[
      \SSg{\{m^{(e,e)}\} \cup \{ \, \id^{s} \mid s \in S \, \}} =
      \SOp(A)\,,
    \]
    where $m^{(e,e)}$ is the binary $S$\hyp{}operation defined by
    $m(x,y):=\max(x,y)\oplus 1$ \textup{(}$\oplus$~denotes addition
    modulo $k$\textup{)} with $\sgn(m)=(e,e)$.

  \item \label{H0-2} The $S$\hyp{}relational clone $\SRel(A)$ is
    finitely generated.

    For instance, for $A=\{0,1,\dots,k-1\}$ and $k\geq 3$, we have
    \begin{align*}
      \SRg{(\Delta,\nabla,\dots,\nabla), (\mathord{\leq}, \mathord{\leq}, \dots, \mathord{\leq}),
      (\mathord{\neq}, \mathord{\neq}, \dots, \mathord{\neq})} = \SRel(A).
    \end{align*}
    Here $(\sigma,\sigma',\dots,\sigma')$ denotes the relation
    $\rho\in\SRel(A)$ with $\rho_{e}=\sigma$ and $\rho_{s}=\sigma'$
    for $s\in S\setminus\{e\}$.

  \item \label{H0-3} The lattice $\SL_{A}$ is atomic and coatomic,
    i.e., each nontrivial $S$\hyp{}preclone contains a minimal one,
    and is contained in a maximal one.
  \end{enumerate}
\end{proposition}

\begin{proof} 
  \ref{H0-1}: It is known (cf., e.g.,~\cite[5.1.4]{PoeK79}), that $m$
  is a so-called Sheffer function, i.e., the clone generated by $m$ is
  the full clone of all operations. Thus $\SSg{m}$ contains \emph{all}
  functions $f\in\SOp(A)$ with $\Sgn(f)=\{e\}$. Let $g\in\SOp(A)$ with
  $\sgn(g)=(s_{1},\dots,s_{n})$ and let $f\in\SOp(A)$ be given by
  $f(x_{1},\dots,x_{n}):=g(x_{1},\dots,x_{x})$ (i.e.,
  $\mathring f=\mathring g$) and $\sgn(f)=(e,\dots,e)$. Then
  $g^{\lambda}(x_{1},\dots,x_{n})=f(\id^{s_{1}}(x_{1}),\dots,\id^{s_{n}}(x_{n}))$
  (cf.\ Definition~\ref{C2}\ref{C2-6}), thus
  $g \in \SSg{\{f\} \cup \{\id^{s_{1}},\dots,\id^{s_{n}}\}} \subseteq
  \SSg{\{m\} \cup \{ \, \id^{s} \mid s \in S \, \}}$.
 Since $g$ was
  chosen arbitrarily, we are done.

  \ref{H0-2}: It is known (cf., e.g., \cite[1.1.22]{PoeK79}), that for
  $k\ge 3$, $\mathring f\in\Pol\{\leq,\neq\}$ implies that
  $\mathring f$ is a projection (because
  $\Rg[A]{\leq,\neq}=\Rel(A)$). In case $k=2$ one must take a ternary
  relation $\sigma$ (see \cite[5.4.5]{PoeK79}). Thus, because of
  Lemma~\ref{D2}\ref{D2a},
  $f\in\SPol\{(\leq,\leq,\dots,\leq),(\neq,\neq,\dots,\neq)\}$ (or
  $f\in\SPol(\sigma,\sigma,\dots,\sigma)$ for $k=2$) implies that
  $\mathring f$ is a projection. From
  $f\Spreserves (\Delta,\nabla,\dots,\nabla)$ we can conclude that all
  arguments with signum $s\neq e$ must be fictitious, i.e., $f$ is a
  trivial projection ($\in\SJ_{A}$). Consequently, by
  Theorem~\ref{Thm-II}, the $S$\hyp{}relational clone generated by the
  above relations equals
  $\SInv \SJ_{A} \stackrel{\ref{D1a}}{=} \SInv\SPol\SRel(A)=\SRel(A)$.

  \ref{H0-3}: It is well known from universal algebra that the
  subalgebra lattice of a finitely generated algebra is coatomic
  (cf.~\cite[proof of 3.1.5]{PoeK79}). Thus it follows from \ref{H0-1}
  that $\SL_{A}$ (as lattice of sub-$S$\hyp{}preclones of $\SOp(A)$)
  is coatomic. Analogously, by \ref{H0-2}, the lattice of all
  $S$\hyp{}relational clones is coatomic. However, due to the Galois
  connection $\SPol$--$\SInv$ and Theorems~\ref{Thm-I} and
  \ref{Thm-II}, the latter lattice is dual to $\SL_{A}$, consequently
  $\SL_{A}$ is atomic.
\end{proof}

\begin{proposition}\label{H0A}
  There are finitely many maximal and finitely many minimal
  $S$\hyp{}preclones in $\SL_{A}$.
\end{proposition}

\begin{proof}
  (a) Let $F\subsetneq\SOp(A)$ be a maximal $S$\hyp{}preclone
  (coatom). We consider the following $S$\hyp{}relations $\rho^{(s)}$,
  $s\in S$: $\rho^{(e)}:=\Gamma_{F}(\chi^{(e,e)})$ and
  $\rho^{(s)}:=\Gamma_{F}(\chi^{(s)})$ for $s\in S\setminus\{e\}$. If
  all these $\rho^{(s)}$ were $S$\hyp{}diagonals, then the functions
  $m^{(e,e)}$ and $\id^{s}$ (see Proposition~\ref{H0}\ref{H0-1}) would
  belong to $F$ by Proposition~\ref{D6X}, i.e., $F$ would contain a
  generating system for $\SOp(A)$, which contradicts the assumption
  $F\subsetneq \SOp(A)$. Thus at least one of the $\rho^{(s)}$, say
  $\rho^{(s_{i})}$, is nontrivial and we have
  $F\subseteq \SPol\SInv
  F\subseteq\SPol\rho^{(s_{i})}\subsetneq\SOp(A)$. By the maximality
  of $F$ we get $F=\SPol\rho^{(s_{i})}$. This means that each maximal
  $S$\hyp{}preclone is determined by an $S$\hyp{}relation of bounded
  arity (more precisely, it is the arity of $\chi^{(e,e)}$ or
  $\chi^{(s)}$, i.e., $|A|^{2}$ or $|A|$). But there exist only
  finitely many $S$\hyp{}relations of fixed arity. Consequently there
  are only finitely many maximal $S$\hyp{}preclones.

  (b) Let $F$ be a minimal $S$\hyp{}preclone. Then $F$ is generated by
  each nontrivial $S$\hyp{}operation $f\in F\setminus \SJ_{A}$. We are
  going to show that there is always a nontrivial function $f'$ in $F$
  with arity at most $m:=|A^{2}|\cdot|S|$. Since there are only
  finitely many functions $f'$ of fixed bounded arity, there are also
  only finitely many minimal $S$\hyp{}preclones (because they are of
  the form $F=\SSg{f'}$), and we are done.

  Let $f\in F\setminus \SJ_{A}$, $\sgn(f)=(s_{1},\dots,s_{n})$. Then
  $f$ does not preserve at least one of the three binary relations
  (denoted by $\rho$ here) in Proposition~\ref{H0}\ref{H0-2}
  generating $\SRel(A)$ (otherwise
  $f\in\SPol\SRel(A) \stackrel{\ref{D1a}}{=} \SJ_{A}$ is trivial). The
  condition $f\notpreserves\rho$ means that there must exist $s\in S$
  and $r_{1}\in\rho_{s_{1}s},\dots,r_{n}\in\rho_{s_{n}s}$ such that
  $f(r_{1},\dots,r_{n})\notin \rho_{s}$. Because $\rho$ is binary,
  $|\rho_{t}|\leq |A^{2}|$ for all $t\in S$, i.e., there exist at most
  $|A^{2}|\cdot|S|$ different elements for $r_{1},\dots,r_{n}$
  ($|A^{2}|$ for each signum $t\in S$). Thus, if $n> m$, one can
  identify arguments in $f(x_{1},\dots,x_{n})$ (namely $x_{i}$ with
  $x_{j}$ if $r_{i}=r_{j}$ and $s_{i}=s_{j}$) and gets a function
  $f'\in \SSg{f}=F$ of arity at most $m$ which still does not preserve
  $\rho$, i.e., $f'\notin \SJ_{A}$ is nontrivial, which was to be
  shown.
\end{proof}

Concerning maximal $S$\hyp{}preclones in $\SL_{A}$ we have the
following characterization.
\begin{proposition}\label{H4}
  Each maximal $S$\hyp{}preclone $F\le\SOp(A)$ can be characterized as
  $F=\SPol \rho$ for some $\rho\in\SRel(A)$ such that we have
  $(^{*})_{s}$ for each $s\in S$, where
  \begin{center}
    $(^{*})_{s}:\iff$ \fbox{\begin{minipage}[c]{0.8\linewidth}
        $\rho_{s}$ is a diagonal relation, or \\
        $\Pol\rho_{s}$ is a maximal clone and
        $\ar(\rho)=\ar(\rho_{s})$ is the minimal arity of a nontrivial
        relation in $[\rho_{s}]=\Inv\Pol\rho_{s}$.
      \end{minipage}}
  \end{center}
\end{proposition}

\begin{proof} Let $F\in \SL_{A}$ be a maximal $S$\hyp{}preclone and
  let $\rho\in\SInv F\setminus \SD_{A}$. Then
  $F\subseteq\SPol\SInv F\subseteq\SPol\rho\subsetneq\SOp(A)$, thus,
  by the maximality of $F$, $F=\SPol\rho$ for each nontrivial
  $S$\hyp{}relation $\rho$ in $\SInv F$.

  Let $F=\SPol\rho$ be maximal and assume $\rho_{s}\notin D_{A}$ (not
  diagonal, not empty) for some $s\in S$. If $(^{*})_{s}$ is not
  satisfied, then (since $\rho_{s}$ is nontrivial) there exists a
  minimal relational clone $[\sigma_{s}]\subseteq[\rho_{s}]$ where we
  choose $\sigma_{s}$ such that $\ar(\sigma_{s})$ is minimal. Thus
  there exists a ``construction'' $t_{\phi}$ (logical operation with
  pp\hyp{}formula $\phi$) such that
  $\sigma_{s}=t_{\phi}(\rho_{s})$. Thus $\rho':=t_{\phi}(\rho)$ is a
  nontrivial $S$\hyp{}relation with $\rho'_{s}=\sigma_{s}$. Moreover
  we have $\rho'\in \SInv F$ (by Lemma~\ref{D3A}), consequently
  $F=\SPol\rho'$ as well. By construction, $\rho'_{s}$ satisfies
  $(^{*})_{s}$. Clearly $\ar(\rho')\leq\ar(\rho)$.

  If $\rho'$ does not satisfy $(^{*})_{s}$ for all $s\in S$, one
  chooses the next $s'\in S$ with nontrivial $\rho'_{s'}$ and repeats
  the above procedure getting $\rho''=t_{\phi'}(\rho')$ with some
  $t_{\phi'}$ such that $t_{\phi'}(\rho'_{s'})=\sigma'_{s'}$ with
  minimal relational clone $[\sigma'_{s'}]$ (and minimal arity
  $\ar(\sigma'_{s'})=\ar(\rho'')\leq\ar(\rho')$). Note that
  $\sigma'_{s}:=t_{\phi'}(\rho'_{s})\in[\rho'_{s}]=[\sigma_{s}]$ and
  therefore either $\sigma'_{s}\in D_{A}$ (this, in particular, is the
  case if $\ar(\rho'')<\ar(\sigma_{s})=\ar(\rho')$) or
  $[\sigma'_{s}]=[\sigma_{s}]$ is minimal (by minimality of
  $[\sigma_{s}]$). Consequently $\rho'':=t_{\phi'}(\rho')$ satisfies
  both $(^{*})_{s}$ and $(^{*})_{s'}$.

  One can continue this until one gets an $S$\hyp{}relation $\omega$
  with $F=\SPol\omega$ satisfying $(^{*})_{s}$ for each $s\in S$
  (i.e., that each $\omega_{s}$ is trivial or $\Pol\omega_{s}$ is a
  maximal clone and $\ar(\omega)$ is the minimal arity of a nontrivial
  relation in $\Rg{\omega_{s}}$).
\end{proof}

\begin{remark}\label{H4a}
  The maximal clones that appear in the conditions $(^{*})_{s}$ are
  known from the classical result of I. Rosenberg in
  \cite{Ros70}. Therefore Proposition~\ref{H4} provides useful
  candidates for determining all maximal $S$\hyp{}preclones. It
  ``only'' remains the task to exclude those $\rho$ which do not give
  a maximal $S$\hyp{}preclone $F=\SPol\rho$ (such as, trivially, all
  diagonal $S$\hyp{}relations). This we shall do in Part~II for
  Boolean $S$\hyp{}preclones (there exist 9 maximal Boolean
  $S$\hyp{}preclones for the two\hyp{}element group $S$).
\end{remark}

At the end of this section we shall deal with the complexity of
$\SL_{A}$ versus $\cL_{A}$ and present two (nearly trivial) embeddings
of the clone lattice $\cL_{A}$ into the $S$\hyp{}preclone lattice
$\SL_{A}$. Moreover we generalize symmetries of $\cL_{A}$ (inner
automorphisms) to $\SL_{A}$.

\begin{proposition}\label{H7} We define
  the mappings $\Psi,\Phi:\cL_{A}\to \SL_{A}$ as follows \textup{(}for
  $F\in\cL_{A}$\textup{)}:
  \begin{align*}
    \Psi(F) &:= \SSg{\{ \, f \in \SOp(A) \mid \mathring{f} \in F \text{ and
              } \Sgn(f) = \{e\} \, \}},\\
    \Phi(F) &:= \{ \, f \in \SOp(A) \mid \mathring{f} \in F \, \}.\
  \end{align*}
  Then $\Psi$ is a lattice embedding into the interval
  $\Int[\SL_{A}]{\SJ_{A},\Psi(\Op(A))}$ in $\SL_{A}$, i.e., into the
  principal ideal generated by $\Psi(\Op(A))$ in $\SL_{A}$, and $\Phi$
  is a lattice embedding into the interval
  $\Int[\SL_{A}]{\Phi(J_{A}),\SOp(A)}$, i.e., into the principal
  filter generated by $\Phi(J_{A})$ in $\SL_{A}$ \textup{(}where
  $\Phi(J_{A}) = \SSg{\{ \, \id^{s} \mid s \in S \,\}}$\textup{)}. If
  $S$ is a group then $\Phi$ is an embedding \textsl{onto} the
  principal filter generated by
  $\SSg{\{ \, \id^{s} \mid s \in S \, \}}$ in $\SL_{A}$.
\end{proposition}

\begin{proof}
  For a clone $F \in \cL_{A}$, $\Psi(F)$ is obtained by taking the
  functions in $F$, giving all arguments the signum $e$, and then
  adding an arbitrary number of fictitious arguments of arbitrary
  signum in all possible ways. Thus each $S$\hyp{}operation in
  $\Psi(F)$ has essential arguments only for signum $e$, in particular
  we have $(\Psi(F)^{[e]})\,\mathring{}=F$ (note that the $e$-part of
  each $S$\hyp{}preclone is a clone by Proposition~\ref{D8}).
 Thus
  $\Psi$ is injective; moreover $\Psi(J_{A})=\SJ_{A}$.

  Clearly, $\Phi(F)$ is an $S$\hyp{}preclone. For a clone
  $F \in \cL_{A}$, $\Phi(F)$ contains, for each operation $f \in F$,
  every $S$\hyp{}operation $g$ with $\mathring{g}=f$ and arbitrary
  signum. Thus $\Phi(\Op(A))=\SOp(A)$ and
  $\Phi(J_{A}) = \SSg{\{ \, \id^{s} \mid s \in S \, \}}$.
 Note also
  that $(\Phi(F))\,\mathring{}=F$.

  From the definitions it easily follows that $\Psi$ and $\Phi$ are
  lattice embeddings (the image of joins (or meets) are joins (or
  meets) of the images).

  If $S$ is a group, then
  $\{ \, \id^{s} \mid s \in S \, \}\subseteq H$ (for some
  $H\in\SL_{A}$) implies the following property for $H$: If
  $f^{\lambda}\in H$ then $f^{\lambda'}\in H$ for arbitrary signa
  $\lambda=(s_{1},\dots,s_{n})$,
  $\lambda'=(s'_{1},\dots,s'_{n})$. This proves $H=\Phi(F)$ with
  $F = \{ \, \mathring{f} \mid f \in H \, \} \in \cL_{A}$ and we are
  done. In fact, for $s_{i}, s_{i}'$ there exists $u_{i}\in S$ such
  that $s_{i}'=s_{i}u_{i}$ ($i\in\{1,\dots,n\}$). Then
  $f^{\lambda'}(x_{1},\dots,x_{n})=f(\id^{u_{1}}(x_{1}),\dots,\id^{u_{n}}(x_{n}))$
  is a composition of $f$ and the $\id^{s}$ (cf.\
  Definition~\ref{C2}\ref{C2-6}) and therefore belongs to $H$ by the
  assumption on $H$.
\end{proof}

The preceding result contains two embeddings of the lattice of clones
into the lattice of $S$\hyp{}preclones. It is worth noting that
Proposition~\ref{D4} defines an inverse of the map $\Phi$ by mapping
each $S$\hyp{}preclone to the clone generated by all operations in the
$S$\hyp{}preclone, ignoring the signa of these operations.

\begin{definition}\label{H5a}

  Let $\pi:A\to A$ be a permutation on $A$. Then, for $f\in\SOp(A)$,
  $\rho=(\rho_{s})_{s\in S}\in\SRela[m](A)$, $F\subseteq\SOp(A)$ and
  $Q\subseteq\SRel(A)$ we define the \New{$\pi$\hyp{}dual} $f^{\pi}$,
  $\rho^{\pi}=(\rho^{\pi}_{s})_{s\in S}$, $F^{\pi}$ and $Q^{\pi}$ as
  follows:
  \begin{align*}
    & f^{\pi}(x_{1},\dots,x_{n}) :=\pi(f(\pi^{-1}x_{1},\dots,\pi^{-1}x_{n})),
      \quad\sgn(f^{\pi}):=\sgn(f),\\
    & \rho^{\pi}_{s}:=\pi\rho_{s} :=
      \{ \, (\pi a_{1},\dots,\pi a_{m}) \mid
      (a_{1},\dots,a_{m})\in\rho_{s} \, \}, \quad s \in S,\\
    & F^{\pi} := \{ \, f^{\pi} \mid f \in F \, \}, \quad
      Q^{\pi} := \{ \, \rho^{\pi} \mid \rho \in Q \, \}. 
  \end{align*}
  An $S$\hyp{}operation $f$ with $f^{\pi}=f$ and an $S$\hyp{}preclone
  $F$ with $F^{\pi}=F$ are called \New{$\pi$\hyp{}selfdual}.
  (\New{$\pi$\hyp{}selfdual} $S$\hyp{}relations and
  $S$\hyp{}relational clones are defined analogously.)

\end{definition}

\begin{remark}\label{H5}
  The mapping $(-)^{\pi}:\SOp(A)\to \SOp(A): f\mapsto f^{\pi}$ is an
  automorphism, called an \New{inner automorphism}, of the full
  $S$\hyp{}preclone $\SOp(A)$ considered as an algebra equipped with
  the operations $\id_{A},\zeta,\tau,\nabla^{s},\Delta,\circ$ (cf.,
  e.g., \cite{Mal1966}, \cite[3.4.1]{PoeK79}).
\end{remark}
For the classical clone $\Op(A)$, there are only inner automorphisms
(\cite[Theorem~2]{Mal1966}). However, for $S$\hyp{}preclones there are
further automorphisms, which are induced by the automorphisms of the
monoid $S$.

\begin{definition}\label{H5b}
  Let $h:S\to S$ be an automorphism of the monoid $S$. For
  $f\in\SOp(A)$, with $\sgn(f)=(s_{1},\dots,s_{n})$ and
  $\rho\in\SRel(A)$ we define $f^{h}$ and $\rho^{h}$ as follows:
  \begin{align*}
    \sgn(f^{h})&:=h(\sgn(f))=(h(s_{1}),\dots, h(s_{n})), \\
    f^{h}(x_{1},\dots,x_{n})&:=f(x_{1},\dots,x_{n}) \quad \text{ (i.e.,
                              ${\mathring{f}}^{h}:=\mathring{f}$),} \\
    \rho^{h}&:=(\rho_{h^{-1}(s)})_{s\in S}, \\
    \text{equivalently, } \;\; \rho^{h}_{h(s)}&:=\rho_{s} \text{ for }s\in S.
  \end{align*}
  Note that $\rho^{h}$ is well defined because $h$ is bijective.

\end{definition}

\begin{remark}\label{H5c}
  The mapping $(-)^{h}:\SOp(A)\to \SOp(A): f\mapsto f^{h}$ changes
  only the signum of $f$ and not the underlying function
  $\mathring{f}$. It is also an automorphism of the full
  $S$\hyp{}preclone
  $\langle\SOp(A);\id_{A},\zeta,\tau,\nabla^{s},\Delta,\circ\rangle$
  (considered as an algebra). This can be checked easily. We show only
  for $\circ$ that indeed $(f\circ g)^{h}=f^{h}\circ g^{h}$ (we have
  to check the signa only; for the notation, see
  Definition~\ref{C2}\ref{C2-6}):
  \begin{align*}
    \sgn((f\circ g)^{h})&=h(\sgn(f\circ g))=
                          (h(s'_{1}s_{1}),\dots,h(s'_{m}s_{1}),h(s_{2}),\dots,h(s_{n}))\\
                        &=(h(s'_{1})h(s_{1}),\dots,h(s'_{m})h(s_{1}),h(s_{2}),\dots,h(s_{n}))\\
                        &=\sgn(f^{h}\circ g^{h}).
  \end{align*}
\end{remark}

The automorphisms also provide inner symmetries of the lattice
$\SL_{A}$.
\begin{proposition}\label{H5d}
  The mappings $(-)^{\pi}:\SL_{A}\to \SL_{A}: F\mapsto F^{\pi}$ and
  $(-)^{h}:\SL_{A}\to \SL_{A}: F\mapsto F^{h}$ are lattice
  automorphisms.
\end{proposition}

\begin{proof}
  It is easy to check that
  $(F_{1}\land F_{2})^{\pi}=(F_{1}\cap F_{2})^{\pi}= F_{1}^{\pi}\cap
  F_{2}^{\pi}=F_{1}^{\pi}\land F_{2}^{\pi}$ and
  $(F_{1}\lor F_{2})^{\pi}= (\SSg{F_{1}\cup
    F_{2}})^{\pi}=\SSg{(F_{1}\cup F_{2})^{\pi}}= \SSg{F_{1}^{\pi}\cup
    F_{2}^{\pi}}=F_{1}^{\pi}\lor F_{2}^{\pi}$, analogously for $h$
  instead of $\pi$.
\end{proof}

Finally, we show the interplay of the ``duality'' operator $(-)^{\pi}$
and the ``signum permuting'' operator $(-)^{h}$ with the Galois
connection $\SPol$--$\SInv$:

\begin{proposition}\label{H6} Let $\pi$ be a permutation on $A$, let
  $h$ be an automorphism of $S$ and $f\in\SOp(A)$, $\rho\in\SRel(A)$,
  $F\subseteq\SOp(A)$, $Q\subseteq\SRel(A)$. Then we have
  \begin{enumerate}[label=\textup{(\roman*)}]
  \item \label{H6-i}
    $f\Spreserves\rho\iff f^{\pi}\Spreserves\rho^{\pi}$, \qquad
    $f\Spreserves\rho\iff f^{h}\Spreserves\rho^{h}$,
  \item \label{H6-ii} $\SInv F^{\pi}=(\SInv F)^{\pi}$, \quad
    $\SInv F^{h}=(\SInv F)^{h}$,
  \item \label{H6-iii} $\SPol Q^{\pi}=(\SPol Q)^{\pi}$, \quad
    $\SPol Q^{h}=(\SPol Q)^{h}$.
  \end{enumerate}
  
\end{proposition}

\begin{proof}
  \ref{H6-i}: Let $\sgn(f)=(s_{1},\dots,s_{n})$ and $s\in S$. Then
  \begin{align*}
    f(\rho_{s_{1}s},\dots,\rho_{s_{n}s})\subseteq \rho_{s} 
    &\iff 
      \pi f(\pi^{-1}\pi\rho_{s_{1}s},\dots,\pi^{-1}\pi\rho_{s_{n}s})\subseteq
      \pi\rho_{s}\\ 
    &\iff 
      \pi f(\pi^{-1}\rho_{s_{1}s}^{\pi},\dots,\pi^{-1}\rho_{s_{n}s}^{\pi})\subseteq
      \rho_{s}^{\pi}\\
    &\iff f^{\pi}(\rho_{s_{1}s}^{\pi},\dots,\rho_{s_{n}s}^{\pi})\subseteq \rho_{s}^{\pi}.
  \end{align*}
  According to \ref{D1}\eqref{D1i} we get the first part of
  \ref{H6-i}.

  Furthermore, with
  $(t_{1},\dots,t_{n}):=(h(s_{1}),\dots,h(s_{n}))=\sgn(f^{h})$ and
  $t:=h(s)$ we have (according to the definitions):
  \begin{align*}
    f(\rho_{s_{1}s},\dots,\rho_{s_{n}s})\subseteq \rho_{s} 
    &\iff 
      f(\rho^{h}_{h(s_{1}s)},\dots,\rho^{h}_{h(s_{n}s)})\subseteq
      \rho^{h}_{h(s)}\\ 
    &\iff 
      f(\rho^{h}_{h(s_{1})h(s)},\dots,\rho^{h}_{h(s_{n})h(s)})\subseteq
      \rho^{h}_{h(s)}\\ 
    &\iff 
      f^{h}(\rho^{h}_{t_{1}t},\dots,\rho^{h}_{t_{n}t})\subseteq
      \rho^{h}_{t}.
  \end{align*}
  According to \ref{D1}\eqref{D1i} we get the second part of
  \ref{H6-i} (note that if $s$ runs through all elements of $S$ then
  also $t=h(s)$ does so, because $h$ is bijective).

  \ref{H6-ii} and \ref{H6-iii} follow directly from \ref{H6-i}.
\end{proof}


\section{Concluding remarks}
\label{secI}

In Section~2 we defined the Galois connection $\SPol$--$\SInv$ for
$S$\hyp{}preclones and $S$\hyp{}relational clones, and
Theorems~\ref{Thm-I}, \ref{Thm-II} show that on a finite set,
$S$\hyp{}preclones correspond exactly to Galois closed sets of
$S$\hyp{}operations and $S$\hyp{}relational clones are captured
precisely by Galois closed sets of $S$\hyp{}relations. Since the
correspondence of clones with relational clones is a fundamental
result that has many applications, one can now ask whether similar
applications hold for $S$\hyp{}preclones and $S$\hyp{}relational
clones. For example:

\begin{problem}\label{Problem1}
  Classify the maximal $S$\hyp{}preclones for a finite monoid $S$ and
  finite set $A$ (analogously to Rosenberg's classification of maximal
  clones \cite{Ros70}).
\end{problem}

The motivating example for $S$\hyp{}preclones is based on the
2\hyp{}element group $S = \{\mathord{+}, \mathord{-}\}$ and the
$S$\hyp{}relation $(\mathord{\leq},\mathord{\geq})$ for a poset
$(A,\mathord{\leq})$ (cf.\ Examples~\ref{Ex0} and
\ref{Ex0a}). $S$\hyp{}preclones for this group are referred to as
$\pm$\hyp{}preclones, and the problem above may be more approachable
if it is first restricted to $\pm$\hyp{}preclones.

Likewise, for any result about clones, one might investigate whether a
similar result holds about $S$\hyp{}preclones.

Proposition~\ref{H0A} shows that for a finite set $A$ there are only
finitely many maximal and finitely many minimal $S$\hyp{}preclones. In
Part II we take a detailed look at the lattice of $\pm$\hyp{}preclones
over a 2\hyp{}element set and find a complete list of atoms and
coatoms of this lattice. In particular, it is shown that
$\SPol\{(\mathord{\leq},\mathord{\geq})\}$ (see Example~\ref{Ex0a}) is
a maximal $\pm$\hyp{}preclone.

Moreover, there appears the question if the notions of
$S$\hyp{}preclone and $S$\hyp{}relational clone can be extended to the
setting where the monoid $S$ of signa is only assumed to be a
semigroup.

Finally it is an interesting problem to develop the theory of
$S$\hyp{}algebras $(A,(f_{i})_{i\in I})$ with fundamental operations
$f_{i}\in \SOp(A)$ for a fixed finite monoid $S$.


\def\cprime{$'$}

\end{document}